\documentclass[a4paper,12pt]{article}
\usepackage{latexsym}
\usepackage{amsmath}
\usepackage{amssymb}
\usepackage{enumerate}

\usepackage{theorem}
\newtheorem{theorem}{Theorem}[section]
\newtheorem{proposition}[theorem]{Proposition}
\newtheorem{lemma}[theorem]{Lemma}
\newtheorem{corollary}[theorem]{Corollary}
\newtheorem{claim}[theorem]{Claim}

\theorembodyfont{\rmfamily}
\newtheorem{proof}{\textmd{\textit{Proof.}}}

\newtheorem{remark}[theorem]{Remark}
\newtheorem{example}[theorem]{Example}
\newtheorem{definition}[theorem]{Definition}

\makeatletter

\@addtoreset{equation}{section}
\makeatother

\newcommand{\qedd}{\hfill \Box}
\newcommand{\ve}{\varepsilon}

\newcommand{\lra}{\longrightarrow}
\newcommand{\grad}{\nabla\!\!_-}
\newcommand{\e}{\mathrm{e}}

\newcommand{\R}{\ensuremath{\mathbb{R}}}

\newcommand{\cC}{\ensuremath{\mathcal{C}}}
\newcommand{\cD}{\ensuremath{\mathcal{D}}}
\newcommand{\cP}{\ensuremath{\mathcal{P}}}
\newcommand{\bJ}{\ensuremath{\mathbf{J}}}
\newcommand{\bT}{\ensuremath{\mathbf{T}}}
\newcommand{\fs}{\ensuremath{\mathfrak{s}}}
\newcommand{\fc}{\ensuremath{\mathfrak{c}}}
\newcommand{\fm}{\ensuremath{\mathfrak{m}}}
\newcommand{\fL}{\ensuremath{\mathfrak{L}}}

\def\id{\mathop{\mathrm{id}}\nolimits}
\def\vol{\mathop{\mathrm{vol}}\nolimits}
\def\loc{\mathop{\mathrm{loc}}\nolimits}
\def\diam{\mathop{\mathrm{diam}}\nolimits}
\def\supp{\mathop{\mathrm{supp}}\nolimits}
\def\trace{\mathop{\mathrm{trace}}\nolimits}
\def\Hess{\mathop{\mathrm{Hess}}\nolimits}
\def\Ric{\mathop{\mathrm{Ric}}\nolimits}
\def\Ent{\mathop{\mathrm{Ent}}\nolimits}
\def\CD{\mathop{\mathrm{CD}}\nolimits}

\def\EVI{\mathop{\mathrm{EVI}}\nolimits}
\def\DC{\mathop{\mathcal{DC}}\nolimits}

\title{$(K,N)$-convexity and the curvature-dimension condition for negative $N$}
\author{Shin-ichi OHTA\thanks{Department of Mathematics, Kyoto University,
Kyoto 606-8502, Japan ({\sf sohta@math.kyoto-u.ac.jp});
Supported in part by the Grant-in-Aid for Young Scientists (B) 23740048.}}
\date{}
\begin{document}

\maketitle

\begin{abstract}
We extend the range of $N$ to negative values in the $(K,N)$-convexity
(in the sense of Erbar--Kuwada--Sturm), the weighted Ricci curvature $\Ric_N$
and the curvature-dimension condition $\CD(K,N)$.
We generalize a number of results in the case of $N>0$ to this setting,
including Bochner's inequality, the Brunn--Minkowski inequality
and the equivalence between $\Ric_N \ge K$ and $\CD(K,N)$.
We also show an expansion bound for gradient flows of Lipschitz $(K,N)$-convex functions.
\end{abstract}

\tableofcontents

\section{Introduction}

The theories of the curvature-dimension condition and the weighted Ricci curvature
are making rapid progress in this decade.
The \emph{curvature-dimension condition} $\CD(K,N)$ of a metric measure space $(X,d,\fm)$
is a kind of convexity condition of an entropy function on the space of probability measures on $X$.
Here $K \in \R$ and $N \in [1,\infty]$ are parameters,
and the simplest case of $\CD(K,\infty)$ is defined by
the $K$-convexity of the relative entropy with respect to $\fm$.
Sometimes $\CD(K,N)$ is regarded as the combination of the lower Ricci curvature bound $\Ric \ge K$
and the upper dimension bound $\dim \le N$,
and this is the case (i.e., equivalent) for Riemannian manifolds equipped with volume measures.
Generally, for Riemannian (and Finsler) manifolds with weighted measures,
$\CD(K,N)$ is equivalent to the lower bound of the \emph{weighted Ricci curvature} $\Ric_N \ge K$.
By a weighted measure we mean a measure $\fm=\e^{-\psi}\vol_g$
on a Riemannian manifold $(M,g)$.
Then it is natural to modify the Ricci curvature by using the weight function $\psi$.
This is how the weighted Ricci curvature $\Ric_N$ shows up,
where the parameter $N$ depends on the property in question.

Recently, a deep progress was made by Erbar, Kuwada and Sturm~\cite{EKS}.
They introduced the \emph{$(K,N)$-convexity} for $K \in \R$ and $N \in (0,\infty)$,
reinforcing the $K$-convexity.
The $(K,N)$-convexity of the relative entropy is called
the \emph{entropic curvature-dimension condition} $\CD^e(K,N)$,
which turns out equivalent to $\CD(K,N)$ on Riemannian manifolds
and has striking applications in the general metric measure setting
including an expansion bound of heat flow,
the Bakry--Ledoux gradient estimate and Bochner's inequality
(\cite[Theorem~7]{EKS}).
This gives a finite-dimensional (i.e., $N<\infty$) counterpart
of Ambrosio, Gigli and Savar\'e's influential work~\cite{AGSrcd},
and there are already a number of fruitful applications and related works
(see \cite{BGG}, \cite{GM}, \cite{HKX}, \cite{Ku}).

The aim of this article is to point out that it is possible and meaningful
to extend the range of $N$ to negative values in these theories
of $(K,N)$-convexity, $\Ric_N$ and $\CD(K,N)$.
The $(K,N)$-convexity (resp.\ $\CD(K,N)$) with $N<0$
is weaker than the $K$-convexity (resp.\ $\CD(K,\infty)$),
thus it covers a wider class of functions (resp.\ spaces).
See Example~\ref{ex:KN} and Corollaries~\ref{cr:weit}, \ref{cr:prod} for some examples.
Admitting $N<0$ in $\Ric_N$ and $\CD(K,N)$ may sound strange
if one sticks to the image that $N$ represents an upper bound of the dimension,
however, its usefulness has already been recognized
in the author's work \cite{OT1}, \cite{OT2} with Takatsu
(see also a related work \cite{Ot} in the PDE theory).
In \cite{OT1} and \cite{OT2}, the convexity of a certain generalization of the relative entropy
(inspired by information theory)
is characterized by the combination of $\Ric_N \ge 0$ and the convexity of another weight function,
and $N$ can be negative (depending on the choice of an entropy).

We briefly explain the contents of the following sections.
In Section~\ref{sc:conv}, we give the definition of $(K,N)$-convex functions
and study their properties, including the \emph{evolution variational inequality}
along gradient curves in the Riemannian setting (Lemma~\ref{lm:EVI}).
In Section~\ref{sc:gf}, we derive some regularizing estimates
from the evolution variational inequality.
We also show an expansion bound for gradient flows of Lipschitz $(K,N)$-convex functions
on Riemannian manifolds (Theorem~\ref{th:exp}).
In Section~\ref{sc:CD}, we introduce $\Ric_N$ and generalize Bochner's inequality
(Theorem~\ref{th:BW}) as well as the Lichnerowicz inequality (Corollary~\ref{cr:Lich}).
Then we define $\CD(K,N)$ and extend the equivalence between
$\CD(K,N)$ and $\Ric_N \ge K$ to $N<0$ (Theorem~\ref{th:CD}).
Finally, we see that the analogue of $\CD^e(K,N)$ implies several functional inequalities.

Although the proofs are parallel to the case of $N>0$ to a large extent,
we give at least sketches for completeness.
Compared to the $N>0$ case, there remain many open questions for $N<0$.
Especially,
\begin{enumerate}[(a)]
\item an expansion bound for general $(K,N)$-convex functions
(guaranteeing the uniqueness of $\EVI_{K,N}$-gradient curves; see Remark~\ref{rm:Lip}),

\item gradient estimates related to $\Ric_N \ge K$ (see Remark~\ref{rm:hf}),

\item a reasonable sufficient condition (or characterization) of $\CD^e(K,N)$
for weighted Riemannian manifolds (see Remark~\ref{rm:CDe})
\end{enumerate}
are important problems
((a) and (b) are closely related via the duality; see \cite{EKS}, \cite{Ku}).

After completing this article, the author learned of Kolesnikov and Milman's
recent work \cite{KM} in which $\Ric_N$ for $N \in (-\infty,0]$ is also considered
(note that $N=0$ is admitted).
By using Bochner's inequality same as \eqref{eq:NBW}
(or the Reilly formula when the boundary is nonempty),
they obtained various Poincar\'e-type inequalities on weighted Riemannian manifolds
(and their boundaries).
See \cite{KM}, \cite{MR} and the references therein for further related works
concerning the ``$N \le 0$'' case on the Euclidean spaces, such as
Borell's convex (or $1/N$-concave) measures
(see \cite{Bo}, \cite{BrLi}, and the paragraph following Theorem~\ref{th:BM})
and a connection with Barenblatt solutions to the porous medium equation
(see \cite{Ot}, \cite{BoLe}).

\bigskip
\noindent\emph{Acknowledgements.}
I am grateful to Kazumasa Kuwada for valuable suggestions and discussions,
especially on the expansion bound in Subsection~\ref{ssc:cont}.
I thank Asuka Takatsu for fruitful discussions,
some of the results in Subsections~\ref{ssc:wRic}, \ref{ssc:CD}
originate from discussions during the joint work \cite{OT1}, \cite{OT2}.
My gratitude also goes to Frank Morgan for drawing my attention to \cite{MR} and \cite{KM},
and to Emanuel Milman for his helpful comments on the background of \cite{MR} and \cite{KM}.

\section{$(K,N)$-convex functions}\label{sc:conv}

We introduce $(K,N)$-convex functions and study their properties on Riemannian manifolds
and then on metric spaces.
We can follow the line of the $N>0$ case in \cite{EKS} (while some inequalities are reversed),
except for Lemma~\ref{lm:sum} in which we have to take care of the ranges of $N_1$ and $N_2$.

\subsection{$(K,N)$-convex functions on Riemannian manifolds}\label{ssc:Riem}

Our Riemannian manifold $(M,g)$ will be always connected, complete,
$\cC^{\infty}$ and without boundary.
Denote by $d_g$ its Riemannian distance.
According to \cite{EKS}, for $K \in \R$ and $N>0$,
we say that a function $f \in \cC^2(M)$ is \emph{$(K,N)$-convex} if
\begin{equation}\label{eq:KN}
\Hess f(v,v) -\frac{\langle \nabla f,v \rangle^2}{N} \ge K|v|^2 \qquad \text{for all}\ v \in TM.
\end{equation}
This reinforces the usual \emph{$K$-convexity} $\Hess f(v,v) \ge K|v|^2$.
We adopt the same definition \eqref{eq:KN} for $N<0$
and shall see that a number of results in \cite{EKS} can be extended,
although it is weaker than the $K$-convexity.

Let $N<0$ throughout the article without otherwise being indicated.
Given $f:M \lra \R$, it is useful to consider the function
\[ f_N(x):=\e^{-f(x)/N}. \]
By calculation, the $(K,N)$-convexity \eqref{eq:KN} is equivalent to
\begin{equation}\label{eq:KN'}
\Hess f_N(v,v) \ge -\frac{K}{N}f_N(x)|v|^2 \qquad \text{for all}\ v \in T_xM,\ x \in M.
\end{equation}
To rewrite the $(K,N)$-convexity in integrated forms, we introduce the functions
\[ \fs_{\kappa}(\theta) := \left\{
 \begin{array}{cl}
 \frac{1}{\sqrt{\kappa}} \sin(\sqrt{\kappa}\theta) & \text{if}\ \kappa>0, \\
 \theta & \text{if}\ \kappa=0, \\
 \frac{1}{\sqrt{-\kappa}} \sinh(\sqrt{-\kappa}\theta) & \text{if}\ \kappa<0,
 \end{array} \right. \quad
\fc_{\kappa}(\theta) := \left\{
 \begin{array}{cl}
 \cos(\sqrt{\kappa}\theta) & \text{if}\ \kappa>0, \\
 1 & \text{if}\ \kappa=0, \\
 \cosh(\sqrt{-\kappa}\theta) & \text{if}\ \kappa<0
 \end{array} \right. \]
for $\kappa \in \R$ and $\theta \ge 0$.
These are solutions to $u''+\kappa u=0$ with the initial conditions
$\fs_{\kappa}(0)=\fc_{\kappa}'(0)=0$ and $\fs_{\kappa}'(0)=\fc_{\kappa}(0)=1$.
We will use the relations
\begin{equation}\label{eq:cs}
\fc_{\kappa}(\theta) =1-2\kappa \fs_{\kappa}\bigg( \frac{\theta}{2} \bigg)^2,
 \qquad \fs_{\kappa}(\theta)
 =2\fs_{\kappa}\bigg( \frac{\theta}{2} \bigg) \fc_{\kappa}\bigg( \frac{\theta}{2} \bigg).
\end{equation}
We also define, for $t \in [0,1]$,
\[ \sigma_{\kappa}^{(t)}(\theta):=\frac{\fs_{\kappa}(t\theta)}{\fs_{\kappa}(\theta)}, \]
where $\theta>0$ if $\kappa \le 0$ and $\theta \in (0,\pi/\sqrt{\kappa})$ if $\kappa>0$.
Set also $\sigma^{(t)}_{\kappa}(0):=t$.

\begin{lemma}\label{lm:KN}
For $f \in \cC^2(M)$, the following are equivalent$:$
\begin{enumerate}[{\rm (i)}]
\item $f$ is $(K,N)$-convex.

\item Along every minimal geodesic $\gamma:[0,1] \lra M$
with $d:=d_g(\gamma(0),\gamma(1))<\pi\sqrt{N/K}$ if $K<0$, we have
\begin{equation}\label{eq:wKN}
f_N\big( \gamma(t) \big) \le \sigma^{(1-t)}_{K/N}(d) f_N\big( \gamma(0) \big)
 +\sigma^{(t)}_{K/N}(d) f_N\big( \gamma(1) \big)
\end{equation}
for all $t \in [0,1]$.

\item Along any nonconstant minimal geodesic $\gamma:[0,1] \lra M$
with $d:=d_g(\gamma(0),\gamma(1))<\pi\sqrt{N/K}$ if $K<0$, we have
\begin{equation}\label{eq:wKN'}
f_N\big( \gamma(1) \big) \ge \fc_{K/N}(d) f_N\big( \gamma(0) \big)
 +\frac{\fs_{K/N}(d)}{d} (f_N \circ \gamma)'(0).
\end{equation}
\end{enumerate}
\end{lemma}

\begin{proof}
The proof is same as \cite[Lemma~2.2]{EKS}.

(i) $\Rightarrow$ (ii):
Denote by $h(t)$ the RHS of \eqref{eq:wKN}, and compare $h''(t)=-(K/N)h(t)d^2$
with \eqref{eq:KN'}.

(ii) $\Rightarrow$ (iii):
This is immediate from $\fs_{K/N}'=\fc_{K/N}$.

(iii) $\Rightarrow$ (i):
For any $v \in T_xM$, applying \eqref{eq:wKN'} to the geodesics $\gamma_{\pm}$
with $\dot{\gamma}_+(0)=v$ and $\dot{\gamma}_-(0)=-v$, we have
\[ f_N\big( \gamma_+(\ve) \big) +f_N\big( \gamma_-(\ve) \big)
 \ge 2\fc_{K/N}(\ve |v|) f_N(x)
 =2 \left\{ 1-\frac{K}{2N}\ve^2 |v|^2 +O(\ve^4) \right\} f_N(x) \]
for small $\ve>0$.
This shows \eqref{eq:KN'}.
$\qedd$
\end{proof}

Notice that \eqref{eq:wKN} does not require the differentiability of $f$.
This leads us to a metric definition of the $(K,N)$-convexity in the next subsection
(see Definition~\ref{df:KN}).

\begin{remark}\label{rm:KN}
In the case of $K<0$, due to the condition $d<\pi\sqrt{N/K}$
coming naturally from the domain of $\sigma^{(t)}_{K/N}$,
\eqref{eq:wKN} and \eqref{eq:wKN'} can control the behavior of $f$
only in balls with radii less than $\pi\sqrt{N/K}$.
\end{remark}

An important advantage in discussing on a Riemannian manifold is the following
\emph{evolution variational inequality} \eqref{eq:EVI}.
We say that a $\cC^1$-curve $\xi:[0,T) \lra M$ is a \emph{gradient curve} of
$f \in \cC^1(M)$ if $\dot{\xi}(t) =-\nabla f(\xi(t))$ for all $t \in (0,T)$.

\begin{lemma}[Evolution variational inequality]\label{lm:EVI}
Let $f \in \cC^1(M)$.
\begin{enumerate}[{\rm (i)}]
\item
If $f$ is $(K,N)$-convex in the sense of \eqref{eq:wKN'},
then every gradient curve $\xi:[0,T) \lra M$ of $f$ enjoys
\begin{equation}\label{eq:EVI}
\frac{d}{dt}\bigg[ \fs_{K/N}\left( \frac{d_g(\xi(t),z)}{2} \right)^2 \bigg]
 +K \fs_{K/N}\left( \frac{d_g(\xi(t),z)}{2} \right)^2
 \le \frac{N}{2}\left\{ 1-\frac{f_N(z)}{f_N(\xi(t))} \right\}
\end{equation}
for all $z \in M$ and almost all $t \in (0,T)$ with $d_g(\xi(t),z)<\pi\sqrt{N/K}$ if $K<0$.

\item
If \eqref{eq:EVI} holds along a $\cC^1$-curve $\xi:[0,T) \lra M$,
then $\xi$ is a gradient curve of $f$.

\item
If \eqref{eq:EVI} holds for all gradient curves $\xi$ of $f$, then $f$ is $(K,N)$-convex.
\end{enumerate}
\end{lemma}

\begin{proof}
The proof is similar to \cite[Lemma~2.4]{EKS}.

(i)
Take $t \in (0,T)$ at where $h(t):=d_g(\xi(t),z)$ is differentiable
(as well as $h(t)<\pi\sqrt{N/K}$ if $K<0$).
Given a minimal geodesic $\gamma:[0,1] \lra M$ from $\xi(t)$ to $z$,
it follows from the first variation formula that
$(h^2/2)'(t)=-\langle \dot{\xi}(t),\dot{\gamma}(0) \rangle$.
To be precise, the first variation formula gives
\[ (h^2/2)_+'(t) \le -\langle \dot{\xi}(t),\dot{\gamma}(0) \rangle, \qquad
 (h^2/2)'_-(t) \ge -\langle \dot{\xi}(t),\dot{\gamma}(0) \rangle \]
($(\cdot)'_+$ and $(\cdot)'_-$ denote the right and left differentiations)
since $\xi(t)$ may be a cut point of $z$,
and then the differentiability of $h$ yields equality.
Thus we have, by \eqref{eq:wKN'} and $\dot{\xi}(t) =-\nabla f(\xi(t))$,
\begin{equation}\label{eq:EVI''}
f_N(z) \ge \fc_{K/N}\big( h(t) \big) f_N\big( \xi(t) \big)
 -\frac{\fs_{K/N}(h(t))}{Nh(t)} f_N\big( \xi(t) \big) \left( \frac{h^2}{2} \right)'(t).
\end{equation}
This is equivalent to \eqref{eq:EVI} by noticing \eqref{eq:cs}.

(ii)
If \eqref{eq:EVI''} holds at $t \in (0,T)$, then we obtain,
given $v \in T_{\xi(t)}M$ and $\gamma(s):=\exp(sv)$,
\[ f_N\big( \gamma(\ve) \big) -\fc_{K/N}(\ve |v|) f_N\big( \xi(t) \big)
 \ge \frac{\fs_{K/N}(\ve |v|)}{N \ve |v|} f_N\big( \xi(t) \big)
 \langle \dot{\xi}(t),\ve v \rangle \]
for small $\ve>0$.
This shows $\langle \nabla f(\xi(t)),v \rangle \ge -\langle \dot{\xi}(t),v \rangle$ for all $v$.
Therefore $\dot{\xi}(t)=-\nabla f(\xi(t))$ for almost all, and hence all $t \in (0,T)$.

(iii)
The last assertion is shown by applying \eqref{eq:EVI''} (instead of \eqref{eq:wKN'})
in the proof of (iii) $\Rightarrow$ (i) in Lemma~\ref{lm:KN}.
$\qedd$
\end{proof}

\begin{example}\label{ex:KN}
The following functions on intervals are $(K,N)$-convex
on their domains (easily checked via \eqref{eq:KN'}):
\begin{enumerate}[(a)]
\item For $K>0$,
\[ f(x)=-N \log\bigg[ \cosh\bigg( x\sqrt{-\frac{K}{N}} \bigg) \bigg],
 \quad x \in \R. \]

\item For $K>0$,
\[ f(x)=-N \log\bigg[ \sinh\bigg( x\sqrt{-\frac{K}{N}} \bigg) \bigg],
 \quad x \in (0,\infty). \]

\item For $K=0$,
\[ f(x)=-N \log x, \quad x \in (0,\infty). \]

\item For $K<0$,
\[ f(x)=-N \log\bigg[ \cos\bigg( x\sqrt{\frac{K}{N}} \bigg) \bigg],
 \quad x \in \bigg( -\frac{\pi}{2}\sqrt{\frac{N}{K}},\frac{\pi}{2}\sqrt{\frac{N}{K}} \bigg). \]
\end{enumerate}
For each of these functions,
we have indeed equality in \eqref{eq:KN'} (and hence in \eqref{eq:KN}).
Therefore, for instance, $f(x)=-N\log x$ is not $K$-convex for any $K \in \R$ (near $x=0$).
\end{example}

\subsection{$(K,N)$-convex functions on metric spaces}\label{ssc:met}

Let $(X,d)$ be a metric space.
A curve $\gamma:[0,1] \lra X$ is called a \emph{minimal geodesic}
if it is minimizing and of constant speed, namely
$d(\gamma(s),\gamma(t))=|s-t|d(\gamma(0),\gamma(1))$ for all $s,t \in [0,1]$.
Given a function $f:X \lra (-\infty,\infty]$,
set $f_N(x):=\e^{-f(x)/N} \in (0,\infty]$ as in the previous subsection
and $\cD[f]:=\{ x \in X \,|\, f(x)<\infty \}$.
The following definition is natural according to Lemma~\ref{lm:KN}.

\begin{definition}[$(K,N)$-convexity]\label{df:KN}
We say that $f:X \lra (-\infty,\infty]$ is \emph{$(K,N)$-convex}
for $K \in \R$ and $N \in (-\infty,0)$ if any pair $x_0,x_1 \in \cD[f]$,
with $d:=d(x_0,x_1)<\pi\sqrt{N/K}$ when $K<0$, admits
a minimal geodesic $\gamma:[0,1] \lra X$ such that
$\gamma(0)=x_0$, $\gamma(1)=x_1$ and
\begin{equation}\label{eq:wKN+}
f_N\big( \gamma(t) \big) \le \sigma^{(1-t)}_{K/N}(d) f_N(x_0)
 +\sigma^{(t)}_{K/N}(d) f_N(x_1)
\end{equation}
for all $t \in [0,1]$.
If \eqref{eq:wKN+} holds along every minimal geodesic, then $f$ is said to be
\emph{strongly $(K,N)$-convex}.
\end{definition}

Notice that $\gamma(t) \in \cD[f]$ and hence $\cD[f]$ is connected,
and that \eqref{eq:wKN+} trivially holds if $x_0 \not\in \cD[f]$ or $x_1 \not\in \cD[f]$.
We remark that the inequality \eqref{eq:wKN+} is reversed for $N>0$.
Let us summarize basic properties of the $(K,N)$-convexity.
Compare the following two lemmas with \cite[Lemmas~2.9, 2.10]{EKS}.

\begin{lemma}\label{lm:scale}
Let $f:X \lra (-\infty,\infty]$ be $(K,N)$-convex.
\begin{enumerate}[{\rm (i)}]
\item
For any $c>0$, the function $cf$ is $(cK,cN)$-convex.
\item
For any $a \in \R$, the function $f+a$ is $(K,N)$-convex.
\end{enumerate}
\end{lemma}

\begin{proof}
These are immediate from the definition and $(cf)_{cN}=f_N$
as well as $(f+a)_N=\e^{-a/N}f_N$.
$\qedd$
\end{proof}

\begin{lemma}[Sum]\label{lm:sum}
Let $K_1,K_2 \in \R$, $N_2>0$ and $N_1<-N_2$.
Assume that $f_1:X \lra (-\infty,\infty]$ is $(K_1,N_1)$-convex
and $f_2:X \lra (-\infty,\infty]$ is strongly $(K_2,N_2)$-convex.
Then the sum $f:=f_1+f_2$ is $(K_1+K_2,N_1+N_2)$-convex.
\end{lemma}

\begin{proof}
Put $K=K_1+K_2$ and $N=N_1+N_2$.
Let us first check that the range where the $(K,N)$-convexity is effective does not
exceed those of the $(K_i,N_i)$-convexities.
There is nothing to prove when $K_1 \ge 0$ and $K_2 \le 0$.

(a) If $K_1<0$ and $K_2 \le 0$, then $N/K \le N/K_1 <N_1/K_1$.

(b) If $K_1 \ge 0$ and $K_2>0$, then the diameter of $\cD[f_2]$ is not greater than $\pi\sqrt{N_2/K_2}$
(see \cite[Remark~2.3]{EKS}).
The strong $(K_2,N_2)$-convexity further shows that,
if there is a maximal pair $x_0,x_1 \in X$ with $d(x_0,x_1)=\pi\sqrt{N_2/K_2}$,
then $x_0 \not\in \cD[f_2]$ or $x_1 \not\in \cD[f_2]$.
Therefore it is enough to consider points $x_0,x_1$ with $d(x_0,x_1)<\pi\sqrt{N_2/K_2}$,
and then the $(K_2,N_2)$-convexity is available between them.

(c) There remains the case where $K_1<0$ and $K_2>0$.
If $N_1/K_1 \ge N_2/K_2$, then the argument in (b) applies.
Thus assume $N_1/K_1<N_2/K_2$.
Then we have
\[ K=K_1+K_2 <\left( 1+\frac{N_2}{N_1} \right) K_1 =\frac{N}{N_1}K_1 <0 \]
and hence $N/K<N_1/K_1$.

Now, by the hypothesis, any pair $x_0,x_1 \in \cD[f]=\cD[f_1] \cap \cD[f_2]$
admits a minimal geodesic $\gamma:[0,1] \lra X$ along which
\begin{align*}
(f_1)_{N_1}\big( \gamma(t) \big) &\le \sigma^{(1-t)}_{K_1/N_1}(d) (f_1)_{N_1}(x_0)
 +\sigma^{(t)}_{K_1/N_1}(d) (f_1)_{N_1}(x_1), \\
(f_2)_{N_2}\big( \gamma(t) \big) &\ge \sigma^{(1-t)}_{K_2/N_2}(d) (f_2)_{N_2}(x_0)
 +\sigma^{(t)}_{K_2/N_2}(d) (f_2)_{N_2}(x_1),
\end{align*}
where $d:=d(x_0,x_1)$.
Thus we have
\begin{align*}
&\log\left[ f_N\big( \gamma(t) \big) \right]
 = -\frac{N_1}{N} \frac{ f_1(\gamma(t))}{N_1}
 -\frac{N_2}{N} \frac{f_2(\gamma(t))}{N_2} \\
&\le \frac{N_1}{N} G_t \left( -\frac{f_1(x_0)}{N_1}, -\frac{f_1(x_1)}{N_1}, \frac{K_1}{N_1}d^2 \right)
 +\frac{N_2}{N} G_t \left( -\frac{f_2(x_0)}{N_2}, -\frac{f_2(x_1)}{N_2}, \frac{K_2}{N_2}d^2 \right),
\end{align*}
where
\[ G_t(\theta,\eta,\kappa):=\log\left[
 \sigma^{(1-t)}_{\kappa}(1)\e^{\theta} +\sigma^{(t)}_{\kappa}(1)\e^{\eta} \right],
 \quad \theta,\eta \in \R,\ \kappa \in (-\infty,\pi^2), \]
and we used $\sigma^{(t)}_{\kappa d^2}(1)=\sigma^{(t)}_{\kappa}(d)$.
The function $G_t$ is convex (\cite[Lemma~2.11]{EKS}) for each fixed $t$,
hence we obtain
\begin{align*}
&G_t \left( -\frac{f_1(x_0)}{N_1}, -\frac{f_1(x_1)}{N_1}, \frac{K_1}{N_1}d^2 \right) \\
&\le -\frac{N_2}{N_1}  G_t \left( -\frac{f_2(x_0)}{N_2}, -\frac{f_2(x_1)}{N_2}, \frac{K_2}{N_2}d^2 \right)
 +\frac{N}{N_1} G_t \left( -\frac{f(x_0)}{N}, -\frac{f(x_1)}{N}, \frac{K}{N}d^2 \right).
\end{align*}
Combining these yields
\[ \log\left[ f_N\big( \gamma(t) \big) \right] \le G_t \left( -\frac{f(x_0)}{N}, -\frac{f(x_1)}{N}, \frac{K}{N}d^2 \right), \]
which completes the proof.
$\qedd$
\end{proof}

\begin{remark}\label{rm:sum}
The summation rule in Lemma~\ref{lm:sum} holds true also for $N_1,N_2>0$ (\cite[Lemma~2.10]{EKS}),
however, fails in the other ranges.
For example, $f_1 \equiv 0$ is $(0,-1)$-convex and $f_2(x)=-2\log x$ is $(0,2)$-convex on $(0,\infty)$,
but the sum $f_1+f_2=f_2$ is not $(0,1)$-convex.
Similarly, $f_1 \equiv 0$ and $f_2(x)=\log x$ are $(0,-1)$-convex,
but their sum is not $(0,-2)$-convex.
\end{remark}

The $(K,N)$-convexity is weaker than the \emph{$K$-convexity}:
\[ f\big( \gamma(t) \big) \le (1-t)f(x_0) +tf(x_1) -\frac{K}{2}(1-t)td^2. \]
More precisely, we have the following with the help of Lemma~\ref{lm:sum}
(similarly to \cite[Lemma~2.12]{EKS}).

\begin{lemma}[Monotonicity]\label{lm:mono}
If $f:X \lra (-\infty,\infty]$ is $(K,N)$-convex,
then it is also $(K',N')$-convex for all $K' \le K$ and $N' \in [N,0)$.

Moreover, if $f$ is $K$-convex, then it is $(K,N)$-convex for all $N<0$.
\end{lemma}

\begin{proof}
The monotonicity in $K$ follows from the fact that $\sigma^{(t)}_{\kappa}(\theta)$
is non-decreasing in $\kappa$ once $t$ and $\theta$ are fixed (see \cite[Remark~2.2]{BS}).
Note also that $\pi\sqrt{N/K'} \le \pi\sqrt{N/K}$ if $K<0$.
The monotonicity in $N$ is a consequence of Lemma~\ref{lm:sum} by letting
\begin{equation}\label{eq:mono}
f_1=f,\quad f_2 \equiv 0,\quad (K_1,N_1)=(K,N),\quad (K_2,N_2)=(0,N'-N).
\end{equation}
The proof of Lemma~\ref{lm:sum} also shows that
\[ -N G_t \left( -\frac{f(x_0)}{N}, -\frac{f(x_1)}{N}, \frac{K}{N}d^2 \right) \]
is non-decreasing in $N \in (-\infty,0)$ once the other quantities are fixed
(use \eqref{eq:mono} again and observe $G_t(0,0,0)=0$).
Then the last assertion follows from
\begin{align*}
&-\lim_{N \to -\infty} N G_t \left( -\frac{f(x_0)}{N}, -\frac{f(x_1)}{N}, \frac{K}{N}d^2 \right) \\
&= \lim_{\ve \downarrow 0} \frac{1}{\ve} \log \left[
 \sigma^{(1-t)}_{-K}(\sqrt{\ve}d) \e^{\ve f(x_0)}
 +\sigma^{(t)}_{-K}(\sqrt{\ve}d) \e^{\ve f(x_1)} \right] \\
&= (1-t)f(x_0) +\frac{K}{6}(1-t)(t^2-2t)d^2 +tf(x_1) +\frac{K}{6}t(t^2-1)d^2 \\
&= (1-t)f(x_0) +tf(x_1) -\frac{K}{2}(1-t)td^2,
\end{align*}
where we used $\sigma^{(t)}_{\kappa}(\theta)=t+(\kappa/6)t(1-t^2)\theta^2 +O(\theta^4)$
(see \cite[Proposition~5.5]{BS}).
$\qedd$
\end{proof}

\section{Gradient flows of $(K,N)$-convex functions}\label{sc:gf}

We continue the study of $(K,N)$-convex functions on a metric space $(X,d)$.
Precisely, we shall employ the evolution variational inequality \eqref{eq:EVI}
as a definition of gradient curves implicitly including the $(K,N)$-convexity
of a potential function (recall Lemma~\ref{lm:EVI}(ii), (iii)),
and derive several regularizing estimates from it.
We also discuss an expansion bound of gradient flows in the Riemannian setting.

\subsection{Gradient flows and evolution variational inequality}\label{ssc:EVI}

Fix $f:X \lra (-\infty,\infty]$ throughout the subsection
and recall $\cD[f]=f^{-1}((-\infty,\infty))$.
In order to give the metric definition of solutions to the gradient flow equation
``$\dot{\xi}=-\nabla f(\xi)$'', we need two notions.
We refer to \cite{AGSbook} for the deep theory of gradient flows in metric spaces.

At $x \in \cD[f]$, define the \emph{local $($descending$)$ slope} of $f$ by
\[ |\grad f|(x) :=\max \left\{ \limsup_{y \to x}\frac{f(x)-f(y)}{d(x,y)},0 \right\}. \]
A curve $\xi:I \lra X$ on an interval $I \subset \R$ is said to be \emph{absolutely continuous}
if there is $h \in L^1_{\loc}(I)$ such that
\begin{equation}\label{eq:abcon}
d\big( \xi(s),\xi(t) \big) \le \int_s^t h(r) \,dr \quad
 \text{for all}\ s,t \in I\ \text{with}\ s<t.
\end{equation}
Then the \emph{metric speed}
\[ |\dot{\xi}|(t) :=\lim_{\delta \to 0} \frac{d(\xi(t),\xi(t+\delta))}{|\delta|} \]
exists at almost every $t \in I$ and gives the minimal function $h$
adapted to \eqref{eq:abcon} (see \cite[Theorem~1.1.2]{AGSbook}).
Absolutely continuous curves are clearly continuous.

\begin{definition}[Gradient curves]\label{df:gf}
Let $\xi:[0,T) \lra X$ be a continuous curve
which is absolutely continuous on $(0,T)$ and $f(\xi(t))<\infty$ for $t \in (0,T)$.
We say that $\xi$ is a \emph{gradient curve} of $f$ if the following
\emph{energy dissipation identity} holds:
\begin{equation}\label{eq:EDI}
f\big( \xi(t) \big) =f\big( \xi(s) \big) -\frac{1}{2}\int_s^t
 \left\{ |\dot{\xi}|(r)^2 +|\grad f|\big( \xi(r) \big)^2 \right\} dr
\end{equation}
for all $0<s<t<T$.
\end{definition}

Motivated by Lemma~\ref{lm:EVI} on Riemannian manifolds,
we also introduce the following elaborate notion of gradient curves.

\begin{definition}[$\EVI_{K,N}$-gradient curves]\label{df:EVI}
Let $\xi:[0,T) \lra X$ be a continuous curve
which is absolutely continuous on $(0,T)$ and $f(\xi(t))<\infty$ for $t \in (0,T)$.
Then, for $K \in \R$ and $N \in (-\infty,0)$,
we say that $\xi$ is an \emph{$\EVI_{K,N}$-gradient curve} of $f$
if the \emph{evolution variational inequality}
\begin{equation}\label{eq:EVI+}
\frac{d}{dt}\bigg[ \fs_{K/N}\left( \frac{d(\xi(t),z)}{2} \right)^2 \bigg]
 +K \fs_{K/N}\left( \frac{d(\xi(t),z)}{2} \right)^2
 \le \frac{N}{2}\left\{ 1-\frac{f_N(z)}{f_N(\xi(t))} \right\}
\end{equation}
holds for all $z \in \cD[f]$ and almost all $t \in (0,T)$
with $d(\xi(t),z)<\pi\sqrt{N/K}$ if $K<0$.
\end{definition}

This is a generalization of the \emph{$\EVI_K$-gradient curve} defined by
\begin{equation}\label{eq:EVIK}
\frac{d}{dt}\bigg[ \frac{d(\xi(t),z)^2}{2} \bigg] +K \frac{d(\xi(t),z)^2}{2}
 \le f(z)-f\big(\xi(t) \big)
\end{equation}
(see \cite{AGSbook}, \cite{DS}),
which is certainly recovered by letting $N \to -\infty$ in \eqref{eq:EVI+}.
Roughly speaking, the existence of $\EVI_{K,N}$-gradient curves
starting from arbitrary starting points implies that the potential function is $(K,N)$-convex
(see Lemma~\ref{lm:EVI}(iii))
and the underlying space is ``Riemannian''
(see \cite{AGSrcd}, and
recall that the inner product was used to obtain \eqref{eq:EVI} from the $(K,N)$-convexity).
The latter implication is related to the contraction property discussed in the next subsection.
The following lemma verifies the consistency in $K$ and $N$
in a similar manner to Lemma~\ref{lm:mono}.

\begin{lemma}[Monotonicity]\label{lm:mono+}
If $\xi:[0,T) \lra X$ is an $\EVI_{K,N}$-gradient curve of $f$,
then it is also an $\EVI_{K',N'}$-gradient curve of $f$ for all $K' \le K$
and $N' \in [N,0)$.

Moreover, if $\xi$ is an $\EVI_K$-gradient curve of $f$,
then it is an $\EVI_{K,N}$-gradient curve of $f$ for all $N<0$.
\end{lemma}

\begin{proof}
The proof is indebted to the estimates same as \cite[Lemma~2.15]{EKS}.
With the help of \eqref{eq:cs}, we can rewrite \eqref{eq:EVI+} in the following two ways:
\begin{align}
\frac{d}{dt}\bigg[ \frac{d(\xi(t),z)^2}{2} \bigg]
&\le \frac{Nd}{\fs_{K/N}(d)} \left\{ 1-\frac{f_N(z)}{f_N(\xi(t))} \right\}
 -2Kd \frac{\fs_{K/N}(d/2)^2}{\fs_{K/N}(d)},
 \label{eq:EVIb} \\
\frac{d}{dt}\bigg[ \frac{d(\xi(t),z)^2}{2} \bigg]
&\le \frac{Nd}{\fs_{K/N}(d)} \left\{ \fc_{K/N}(d)-\frac{f_N(z)}{f_N(\xi(t))} \right\},
 \label{eq:EVIa}
\end{align}
where we set $d:=d(\xi(t),z)$ in the RHS and assume $d>0$.

One sees the monotonicity in $K$ by \eqref{eq:EVIa} and the fact that
$\fs_{K/N}(d)$ and $\fc_{K/N}(d)/\fs_{K/N}(d)$ are increasing in $K$.
The monotonicity in $N$ follows from \eqref{eq:EVIb} since the functions
\[ \frac{N}{\fs_{K/N}(d)}(1-\e^{a/N}), \qquad
 -K\frac{\fs_{K/N}(d/2)^2}{\fs_{K/N}(d)}
 =-\frac{K}{2} \frac{\fs_{K/N}(d/2)}{\fc_{K/N}(d/2)} \]
are non-decreasing in $N \in (-\infty,0)$ for each fixed $a \in \R$.
The last assertion is a consequence of the above monotonicity of the RHS of \eqref{eq:EVIb}
in $N$ together with the convergence of \eqref{eq:EVIb} to \eqref{eq:EVIK} as $N \to -\infty$.
$\qedd$
\end{proof}

It is now well known that $\EVI_K$-gradient curves enjoy several useful estimates.
We can generalize some of them to $\EVI_{K,N}$-gradient curves,
though $\EVI_{K,N}$ is weaker than $\EVI_K$.
Compare the following propositions and corollary with \cite[Propositions~2.17, 2.18]{EKS}.

\begin{proposition}\label{pr:EVIgf}
Let $\xi:[0,T) \lra X$ be an $\EVI_{K,N}$-gradient curve of $f$ such that
\begin{enumerate}[{\rm (1)}]
\item $\xi$ is locally Lipschitz on $(0,T)$,
\item $f \circ \xi$ is locally bounded above on $(0,T)$.
\end{enumerate}
Then $\xi$ is a gradient curve of $f$ also in the sense of Definition~$\ref{df:gf}$.
In particular, $f(\xi(t))$ is non-increasing in $t$.
\end{proposition}

\begin{proof}
We can follow the line of \cite[Proposition~4.6]{AG} concerning $\EVI_K$.
Fix $t \in (0,T)$ where \eqref{eq:EVI+} holds.
We first observe from the triangle inequality that
\[ \frac{d}{dt}\bigg[ \frac{d(\xi(t),z)^2}{2} \bigg]
 \ge -|\dot{\xi}|(t) d\big( \xi(t),z \big). \]
This and \eqref{eq:EVIa} imply, by abbreviating $d:=d(\xi(t),z)$,
\[ -\frac{\fs_{K/N}(d)}{N}|\dot{\xi}|(t)
 \ge \frac{\fs_{K/N}(d)}{Nd} \frac{d}{dt}\bigg[ \frac{d(\xi(t),z)^2}{2} \bigg]
 \ge \frac{1}{f_N(\xi(t))} \left\{ \fc_{K/N}(d)f_N\big( \xi(t) \big) -f_N(z) \right\} \]
(for $z$ close to $\xi(t)$ if $K<0$).
Dividing by $d$ and letting $z \to \xi(t)$, we obtain
\begin{equation}\label{eq:mgf}
\frac{|\grad f_N|(\xi(t))}{f_N(\xi(t))} \le -\frac{1}{N}|\dot{\xi}|(t), \qquad
 |\grad f|\big( \xi(t) \big) =-\frac{N}{f_N(\xi(t))} |\grad f_N|\big( \xi(t) \big)
 \le |\dot{\xi}|(t).
\end{equation}

In order to estimate $(f \circ \xi)'(t)$,
we deduce from the above calculation with $z=\xi(s)$ for $s$ close to $t$ that
\[ \fc_{K/N}\big( d\big(\xi(s),\xi(t) \big) \big) f_N\big( \xi(t) \big) -f_N\big( \xi(s) \big)
 \le -\frac{\fs_{K/N}(d(\xi(s),\xi(t)))}{N} f_N\big( \xi(t) \big)
 |\dot{\xi}|(t). \]
Since $f_N(\xi(t))$ and $|\dot{\xi}|(t)$ are locally bounded in $t$
by the hypotheses (1) and (2),
we find that $f_N \circ \xi$ is locally Lipschitz on $(0,T)$.
Now, integrate \eqref{eq:EVI+} to obtain for $\delta>0$
\begin{align*}
&\fs_{K/N} \left( \frac{d(\xi(t+\delta),\xi(t))}{2} \right)^2 \\
&\le \frac{N}{2} \int_t^{t+\delta} \left\{ 1-\frac{f_N(\xi(t))}{f_N(\xi(s))} \right\} ds
 -K \int_t^{t+\delta} \fs_{K/N} \left( \frac{d(\xi(s),\xi(t))}{2} \right)^2 ds \\
&= \frac{N}{2} \int_t^{t+\delta} \frac{f_N(\xi(s))-f_N(\xi(t))}{f_N(\xi(s))} \,ds +O(\delta^3).
\end{align*}
Dividing by $\delta^2$ and letting $\delta \downarrow 0$ gives
\[ \frac{|\dot{\xi}|(t)^2}{4}
 \le \frac{N}{4} \frac{(f_N \circ \xi)'(t)}{f_N(\xi(t))}
 = -\frac{(f \circ \xi)'(t)}{4}. \]
Combining this with \eqref{eq:mgf}, we conclude that
\[ (f \circ \xi)'(t) \le -|\dot{\xi}|(t)^2
 \le -\frac{1}{2} \left\{ |\dot{\xi}|(t)^2 +|\grad f|\big( \xi(t) \big)^2 \right\} \]
holds for almost all $t \in (0,T)$.
Integrating this inequality shows the desired identity \eqref{eq:EDI}
since the reverse inequality is readily verified by the local Lipschitz continuity of
$\xi$ and $f \circ \xi$.
$\qedd$
\end{proof}

\begin{remark}\label{rm:EVIgf}
In the case of $N=\infty$, the assumptions (1), (2) in the above proposition are superfluous
since they are consequences of \eqref{eq:EVIK}.
It is unclear (to the author) if (1) and (2) can be removed for general $N \in (-\infty,0)$ or not.
Notice that \eqref{eq:EVIK} immediately implies (2).
The key ingredient for verifying (1) is an expansion bound of the gradient flow
(see \cite[Proposition~4.6]{AG}), however,
at present we can show it only under the Lipschitz continuity of
potential functions when $N \in (-\infty,0)$ (see Theorem~\ref{th:exp}).
\end{remark}

\begin{proposition}\label{pr:EVI1}
Let $\xi:[0,T) \lra X$ be a continuous curve which is locally Lipschitz on $(0,T)$
and $f(\xi(t))<\infty$ for $t \in (0,T)$.
Then, for $K \in \R$ and $N<0$, $\xi$ is an $\EVI_{K,N}$-gradient curve of $f$
if and only if
\begin{equation}\label{eq:EVI'}
\frac{N(\e^{K(t_1-t_0)}-1)}{2K} \left\{ 1-\frac{f_N(z)}{f_N(\xi(t_1))} \right\}
 \ge \e^{K(t_1-t_0)} \fs_{K/N}\left( \frac{d(\xi(t_1),z)}{2} \right)^2
 -\fs_{K/N}\left( \frac{d(\xi(t_0),z)}{2} \right)^2
\end{equation}
holds for all $z \in \cD[f]$ and $0 \le t_0 \le t_1<T$ with
$\sup_{t \in [t_0,t_1]}d(\xi(t),z)<\pi\sqrt{N/K}$ if $K<0$.
\end{proposition}

When $K=0$, $\{\e^{K(t_1-t_0)}-1\}/K$ in the LHS of \eqref{eq:EVI'} is read as $t_1-t_0$.
Notice that $\{\e^{K(t_1-t_0)}-1\}/K$ is nonnegative for all $K \in \R$.

\begin{proof}
Observe that \eqref{eq:EVI+} is equivalent to
\[ \frac{d}{dt}\bigg[ \e^{Kt} \fs_{K/N}\left( \frac{d(\xi(t),z)}{2} \right)^2 \bigg]
 \le \frac{N\e^{Kt}}{2} \left\{ 1-\frac{f_N(z)}{f_N(\xi(t))} \right\}. \]
If $\xi$ is an $\EVI_{K,N}$-gradient curve, then $f_N \circ \xi$ is non-increasing
by Proposition~\ref{pr:EVIgf} and hence we have by integration
\begin{equation}\label{eq:EVI'+}
\e^{Kt_1} \fs_{K/N}\left( \frac{d(\xi(t_1),z)}{2} \right)^2
 -\e^{Kt_0} \fs_{K/N}\left( \frac{d(\xi(t_0),z)}{2} \right)^2
 \le \frac{N(\e^{Kt_1}-\e^{Kt_0})}{2K} \left\{ 1-\frac{f_N(z)}{f_N(\xi(t_1))} \right\},
\end{equation}
where $(\e^{Kt_1}-\e^{Kt_0})/K$ is read as $t_1-t_0$ if $K=0$.
This is equivalent to \eqref{eq:EVI'}.
The converse implication is immediate by dividing \eqref{eq:EVI'+} by $t_1-t_0$
and letting $t_0 \to t_1$.
$\qedd$
\end{proof}

\begin{corollary}\label{cr:EVI}
Let $\xi:[0,T) \lra X$ be an $\EVI_{K,N}$-gradient curve of $f$
which is locally Lipschitz on $(0,T)$.
Then the following hold$:$

\begin{enumerate}[{\rm (i)}]
\item We have the \emph{uniform regularizing bound:}
\[ \frac{f_N(z)}{f_N(\xi(t))} \ge 1+\frac{2K}{N(\e^{Kt}-1)}
 \fs_{K/N}\left( \frac{d(\xi(0),z)}{2} \right)^2 \]
for all $z \in \cD[f]$ and $t \in (0,T)$ with
$\sup_{s \in [0,t]}d(\xi(s),z)<\pi\sqrt{N/K}$ if $K<0$.

\item If $f$ is bounded below, then we have the \emph{uniform continuity estimate:}
\[ \fs_{K/N}\left( \frac{d(\xi(t_0),\xi(t_1))}{2} \right)^2
 \le \frac{N(1-\e^{K(t_0-t_1)})}{2K} \left\{ 1-\frac{f_N(\xi(t_0))}{\inf_X f_N} \right\} \]
for all $0<t_0 \le t_1<T$ with $\sup_{t \in [t_0,t_1]}d(\xi(t),\xi(t_0))<\pi\sqrt{N/K}$ if $K<0$.
\end{enumerate}
\end{corollary}

\begin{proof}
(i) Let $t_0=0$ and $t_1=t$ in \eqref{eq:EVI'}.

(ii) Let $z=\xi(t_0)$ in \eqref{eq:EVI'}.
$\qedd$
\end{proof}

\subsection{An expansion bound for gradient flows of Lipschitz $(K,N)$-convex functions}\label{ssc:cont}

The expansion bound (also called the contraction property)
is a key tool for analyzing gradient flows of convex functions.
In the $N>0$ case, it was shown in \cite[Theorem~2.19]{EKS} that
the evolution variational inequality $\EVI_{K,N}$ implies an expansion bound
without the Lipschitz continuity assumption on potential functions.

Although we will argue on Riemannian manifolds,
the key ingredient is a kind of evolution variational inequality \eqref{eq:EVI?}
which makes sense also in the metric measure setting.
We remark that \eqref{eq:EVI?} is a global inequality,
while \eqref{eq:EVI} is not global when $K<0$.

\begin{theorem}\label{th:exp}
Let $f:M \lra \R$ be a Lipschitz $(K,N)$-convex function on a Riemannian manifold $(M,g)$
such that $N<0$ and $|\nabla f| \le L$ almost everywhere.
Then, given any $x,y \in M$ and the gradient curves $\xi, \zeta:[0,\infty) \lra M$ of $f$
with $\xi(0)=x, \zeta(0)=y$, we have
\begin{equation}\label{eq:exp}
d_g\big( \xi(t_0),\zeta(t_1) \big)^2 \le 2\e^{-\Theta(t_0,t_1)}
 \left\{ \frac{d_g(x,y)^2}{2} -\frac{N(\sqrt{t_1}-\sqrt{t_0})^2}{\Theta(t_0,t_1)}
 (\e^{\Theta(t_0,t_1)}-1 ) \right\}
\end{equation}
for all $t_0,t_1>0$, where we set
\[ \Theta(t_0,t_1) =\Theta_{K,N,L}(t_0,t_1)
 :=\left( 2K+\frac{4L^2}{N} \right) \frac{t_1+\sqrt{t_1 t_0}+t_0}{3} \]
and $(\e^{\Theta(t_0,t_1)}-1)/\Theta(t_0,t_1)$ is read as $1$ if $\Theta(t_0,t_1)=0$.
\end{theorem}

\begin{proof}
We first show that $f$ is $(K+L^2/N)$-convex,
which yields that $\xi$ and $\zeta$ are uniquely determined and Lipschitz.

\begin{claim}\label{cl:Lip}
$f$ is $(K+L^2/N)$-convex.
\end{claim}

\begin{proof}
Fix a unit speed minimal geodesic $\gamma:[0,l] \lra M$ and $t \in (0,l)$
at where $f \circ \gamma$ is differentiable.
Similarly to Lemma~\ref{lm:KN},
it follows from the $(K,N)$-convexity of $f$ that
\[ f_N\big( \gamma(t+\ve) \big) +f_N\big( \gamma(t-\ve) \big)
 \ge 2\bigg\{ 1-\frac{K}{2N}\ve^2 +O_{K,N}(\ve^4) \bigg\} f_N\big( \gamma(t) \big). \]
The LHS is expanded as
\begin{align*}
&\e^{-f(\gamma(t))/N} \bigg\{ 1+\frac{f(\gamma(t))-f(\gamma(t+\ve))}{N}
 +\frac{1}{2} \bigg( \frac{f(\gamma(t))-f(\gamma(t+\ve))}{N} \bigg)^2 +O_{N,L}(\ve^3) \bigg\} \\
&+ \e^{-f(\gamma(t))/N} \bigg\{ 1+\frac{f(\gamma(t))-f(\gamma(t-\ve))}{N}
 +\frac{1}{2} \bigg( \frac{f(\gamma(t))-f(\gamma(t-\ve))}{N} \bigg)^2 +O_{N,L}(\ve^3) \bigg\} \\
&\le \e^{-f(\gamma(t))/N} \bigg\{ 2+\frac{2f(\gamma(t))-f(\gamma(t+\ve))-f(\gamma(t-\ve))}{N}
 +\frac{L^2}{N^2}\ve^2 +O_{N,L}(\ve^3) \bigg\}.
\end{align*}
Thus we have
\[ f\big( \gamma(t+\ve) \big) +f\big( \gamma(t-\ve) \big) -2f\big( \gamma(t) \big)
 \ge \bigg( K+\frac{L^2}{N} \bigg) \ve^2 +O_{K,N,L}(\ve^3), \]
where $O_{K,N,L}(\ve^3)$ depends only on $K,N$ and $L$.
Hence $f$ is $(K+L^2/N)$-convex.
$\hfill \diamondsuit$
\end{proof}

Put $u(s):=d_g(\xi(st_0),\zeta(st_1))^2/2$ and fix $s \in (0,1)$ such that
$u$, $\xi$ and $\zeta$ are differentiable at $s$, $st_0$ and $st_1$, respectively,
and that $(f \circ \xi)'_+(st_0)=-|\grad f|(\xi(st_0))^2$
and $(f \circ \zeta)'_+(st_1)=-|\grad f|(\zeta(st_1))^2$ hold.
Let $\gamma:[0,1] \lra M$ be a minimal geodesic from $\xi(st_0)$ to $\zeta(st_1)$.
Then it follows from the first variation formula that
\[ u'(s) =t_1 \langle \dot{\zeta}(st_1),\dot{\gamma}(1) \rangle
 -t_0 \langle \dot{\xi}(st_0),\dot{\gamma}(0) \rangle
 \le -t_1(f \circ \gamma)'_-(1) +t_0(f \circ \gamma)'_+(0), \]
where the latter inequality holds since $\xi$ and $\zeta$ are gradient curves of $f$
(see \cite[Lemma~4.2]{Ogra} for instance).
Notice that $f \circ \gamma$ is twice differentiable almost everywhere
since $f$ is $(K+L^2/N)$-convex.
Thus, by interpolating $t_{\tau}:=\{(1-\tau)\sqrt{t_0}+\tau\sqrt{t_1}\}^2$
between $t_0$ and $t_1$, we deduce from the $(K,N)$-convexity of $f$ that
\begin{align}
&-t_1(f \circ \gamma)'_-(1) +t_0(f \circ \gamma)'_+(0)
 \le -\int_0^1 \frac{d}{d\tau}\left[ t_{\tau} (f \circ \gamma)'(\tau) \right] d\tau \nonumber\\
&\le -\int_0^1 \left\{ \dot{t}_{\tau} (f \circ \gamma)'(\tau)
 +t_{\tau} \left( K|\dot{\gamma}(\tau)|^2 +\frac{(f \circ \gamma)'(\tau)^2}{N} \right) \right\} d\tau.
 \label{eq:EVI?}
\end{align}
Rewrite the RHS and estimate it by the Lipschitz continuity as
\begin{align*}
&-\int_0^1 \left\{ \dot{t}_{\tau} (f \circ \gamma)'(\tau)
 -\frac{t_{\tau}}{N}(f \circ \gamma)'(\tau)^2
 +t_{\tau} \left( K|\dot{\gamma}(\tau)|^2+\frac{2(f \circ \gamma)'(\tau)^2}{N} \right) \right\} d\tau \\
&\le -\frac{N}{4} \int_0^1 \frac{(\dot{t}_{\tau})^2}{t_{\tau}} \,d\tau
 -\int_0^1 t_{\tau} \left( K+\frac{2L^2}{N} \right) |\dot{\gamma}(\tau)|^2 \,d\tau.
\end{align*}
In the RHS, we calculate
\[ \int_0^1 \frac{(\dot{t}_{\tau})^2}{t_{\tau}} \,d\tau=4(\sqrt{t_1}-\sqrt{t_0})^2, \qquad
 \int_0^1 t_{\tau} \,d\tau=\frac{t_1+\sqrt{t_1 t_0}+t_0}{3}. \]
Thus we obtain
\[ u'(s) \le -N(\sqrt{t_1}-\sqrt{t_0})^2 -\Theta(t_0,t_1) u(s). \]
This implies that
\[ \e^{s\Theta(t_0,t_1)} u(s)
 +\frac{N(\sqrt{t_1}-\sqrt{t_0})^2}{\Theta(t_0,t_1)}(\e^{s\Theta(t_0,t_1)}-1) \]
is non-increasing in $s$, then \eqref{eq:exp} immediately follows.
$\qedd$
\end{proof}

Choosing the same time $t_0=t_1=:t$ in \eqref{eq:exp} yields
\[ d_g\big( \xi(t),\zeta(t) \big) \le \e^{-(K+2L^2/N)t} d_g(x,y). \]
This is slightly worse than the bound $d_g(\xi(t),\zeta(t)) \le \e^{-(K+L^2/N)t}d_g(x,y)$
directly derived from the $(K+L^2/N)$-convexity of $f$.
In either bound, letting $N \to -\infty$ recovers the \emph{$K$-contraction property}:
\[ d_g\big( \xi(t),\zeta(t) \big) \le \e^{-Kt} d_g(x,y). \]
It is essential to discuss on ``Riemannian'' spaces,
otherwise the $K$-convexity does not necessarily imply the $K$-contraction property
(see \cite{OSnc} for an investigation on Finsler manifolds).

\begin{remark}\label{rm:Lip}
(a)
In general, the $L$-Lipschitz continuity of a potential function gives the immediate bound:
\[ d\big( \xi(t),\zeta(t) \big) \le d\big( \xi(0),\zeta(0) \big) +2Lt. \]
We remark that, however, even the uniqueness of gradient curves fails
for general Lipschitz functions
(see \cite[Example~4.23]{AG} for a simple example in the $\ell^2_{\infty}$-space).

(b)
The expansion bound in \cite[Theorem~2.19]{EKS} is, for $K=0$ and $N>0$,
\[ d\big( \xi(t_0),\zeta(t_1) \big)^2 \le d(x,y)^2 +2N(\sqrt{t_1}-\sqrt{t_0})^2 \]
(see also \cite{BGL}).
Obviously this inequality can not be extended to $N<0$ since it is stronger than that for $N>0$.

(c)
Under Bochner's inequality \eqref{eq:NBW} with $N \ge 1$
(the \emph{analytic curvature-dimension condition \`a la Bakry--\'Emery}),
another dimension dependent contraction property for heat semigroup
in terms of the \emph{Markov transportation distance} follows from \cite[Theorem~4.5]{BGG}.
This contraction is different from the one in \cite{EKS} and seems to make sense also for $N<0$,
whereas the author does not know if it can be extended to $N<0$.
\end{remark}

\section{Curvature-dimension condition}\label{sc:CD}

We switch to the related subject of curvature-dimension condition.
We first define the weighted Ricci curvature $\Ric_N$ followed by associated Bochner's inequality.
Then we introduce the original, reduced and entropic curvature-dimension conditions
and discuss their applications.

\subsection{Weighted Ricci curvature}\label{ssc:wRic}

Let $(M,g)$ be an $n$-dimensional Riemannian manifold with $n \ge 2$.
We denote the Riemannian volume measure by $\vol_g$
and fix a weighted measure $\fm=\e^{-\psi}\vol_g$ with $\psi \in \cC^{\infty}(M)$.
Then the Laplacian and Ricci curvature are modified into
$\Delta_{\fm} u:=\Delta u-\langle \nabla u,\nabla\psi \rangle$ and
\[ \Ric_N(v):=\Ric(v)+\Hess\psi(v,v)-\frac{\langle \nabla\psi,v \rangle^2}{N-n} \]
for $v \in TM$.
The parameter $N$ had been usually chosen from $[n,\infty]$,
and the bound $\Ric_N(v) \ge K|v|^2$ is known to imply
many analytic and geometric consequences corresponding to $\Ric \ge K$ as well as $\dim \le N$
(see \cite{Qi}, \cite{Lo}).
The generalization admitting negative values $N<0$
appeared and turned out meaningful in \cite{OT1} and \cite{OT2}.
We will fix $N<0$ as in the previous sections.
Letting $N \to -\infty$ in $\Ric_N$ recovers the \emph{Bakry--\'Emery tensor}
$\Ric +\Hess\psi$ which is usually regarded as $\Ric_{\infty}$.

Let us give applications of $\Ric_N$ with $N<0$
before discussing the curvature-dimension condition.
From the \emph{Bochner--Weitzenb\"ock formula} for $\Ric_{\infty}$:
\begin{equation}\label{eq:BW}
\Delta_{\fm}\left( \frac{|\nabla u|^2}{2} \right) -\langle \nabla \Delta_{\fm}u,\nabla u \rangle
 =\Ric_{\infty}(\nabla u) +\|\!\Hess u\|^2
\end{equation}
($\|\cdot\|$ denotes the Hilbert--Schmidt norm),
we can derive the following inequality similarly to the case of $N \in [n,\infty]$.

\begin{theorem}[Bochner's inequality]\label{th:BW}
For any $u \in \cC^{\infty}(M)$ and $N<0$, we have
\begin{equation}\label{eq:NBW}
\Delta_{\fm}\left( \frac{|\nabla u|^2}{2} \right) -\langle \nabla \Delta_{\fm}u,\nabla u \rangle
 \ge \Ric_N(\nabla u) +\frac{(\Delta_{\fm}u)^2}{N}.
\end{equation}
\end{theorem}

\begin{proof}
This is done by calculation similarly to the case of $N \ge n$,
the details can be found in \cite[Theorem~3.3]{OSbw} for example.
Let $B$ be the matrix representation of $\Hess u$ in an orthonormal coordinate.
Since $B$ is symmetric, we have
\[ \|\!\Hess u\|^2 =\trace(B^2) \ge \frac{(\trace B)^2}{n} =\frac{(\Delta u)^2}{n}. \]
Note that $\Delta u=\Delta_{\fm}u +\langle \nabla u,\nabla\psi \rangle$ and,
for any $a,b \in \R$,
\[ \frac{(a+b)^2}{n}
 =\frac{a^2}{N}-\frac{b^2}{N-n}+\frac{N(N-n)}{n}\left( \frac{a}{N}+\frac{b}{N-n} \right)^2
 \ge \frac{a^2}{N}-\frac{b^2}{N-n} \]
(notice that the inequality fails for $N \in (0,n)$).
Applying this inequality to $a=\Delta_{\fm} u$ and $b=\langle \nabla u,\nabla\psi \rangle$
completes the proof.
$\qedd$
\end{proof}

One can readily obtain a generalization of the Lichnerowicz inequality from \eqref{eq:NBW}.

\begin{corollary}[Lichnerowicz inequality]\label{cr:Lich}
Let $M$ be compact and satisfy $\Ric_N \ge K$ for $K>0$ and $N<0$.
Then the first nonzero eigenvalue of the nonnegative operator $-\Delta_{\fm}$
is bounded from below by $KN/(N-1)$.
\end{corollary}

\begin{proof}
For any $u \in \cC^{\infty}(M)$,
we deduce from \eqref{eq:NBW} and the integration by parts that
\[ \left( 1-\frac{1}{N} \right) \int_M (\Delta_{\fm} u)^2 \,d\fm
 \ge \int_M \Ric_N(\nabla u) \,d\fm \ge K\int_M |\nabla u|^2 \,d\fm. \]
Hence, for an eigenfunction $u$ with $\Delta_{\fm}u=-\lambda u$, we have
\[ \lambda \int_M |\nabla u|^2 \,d\fm =-\lambda \int_M u\Delta_{\fm}u \,d\fm
 =\int_M (\Delta_{\fm}u)^2 \,d\fm \ge \frac{KN}{N-1} \int_M |\nabla u|^2 \,d\fm. \]
This completes the proof.
$\qedd$
\end{proof}

\begin{remark}[Finsler case]\label{rm:FBW}
The weighted Ricci curvature for Finsler manifolds was introduced in \cite{Oint}
and the analogues of the Lichnerowicz inequality,
Bochner--Weitzenb\"ock formula \eqref{eq:BW} and Bochner's inequality \eqref{eq:NBW}
for $N \in [n,\infty]$ were obtained in \cite{Oint} and \cite{OSbw}
along with gradient estimates as applications
(see also \cite{OShf} for a preceding analytic study of heat flow).
One can similarly extend \eqref{eq:NBW} with $N<0$ to the Finsler setting.
The Bochner--Weitzenb\"ock formula was recently further generalized
to Hamiltonian systems in \cite{Oham} with the help of \cite{Le}.
\end{remark}

\subsection{Original and reduced curvature-dimension conditions}\label{ssc:CD}

The theory of convex functions and the Ricci curvature
are connected by the curvature-dimension condition.
The curvature-dimension condition is a convexity condition of an entropy function
on the space of probability measures,
and characterizes lower Ricci curvature bounds for Riemannian (or Finsler) manifolds.
We shall give the precise definition in the sense of Sturm~\cite{StI}, \cite{StII},
see also \cite[Part~III]{Vi} for background and applications.

Let $(X,d)$ be a complete, separable metric space.
Denote by $\cP(X)$ the set of all Borel probability measures on $X$,
and by $\cP^2(X) \subset \cP(X)$ the subset consisting of measures of finite second moments.
For $\mu,\nu \in \cP^2(X)$, the \emph{$L^2$-Wasserstein distance} is defined by
\[ W_2(\mu,\nu):=\inf_{\pi \in \Pi(\mu,\nu)}
 \left( \int_{X \times X} d(x,y)^2 \,\pi(dxdy) \right)^{1/2}, \]
where $\Pi(\mu,\nu) \subset \cP(X \times X)$ is the set of all couplings of $\mu$ and $\nu$.
A coupling attaining the above infimum is called an \emph{optimal coupling}.

Let us fix a Borel measure $\fm$ on $X$.
For $\mu \in \cP(X)$, we define the (relative) \emph{R\'enyi entropy} with respect to $\fm$ by
\[ S_N(\mu):=\int_X \rho^{(N-1)/N} \,d\fm \]
if $\mu$ is absolutely continuous with respect to $\fm$ ($\mu \ll \fm$),
$S_N(\mu):=\infty$ otherwise.
We suppressed the dependence on $\fm$ for notational simplicity.
The R\'enyi entropy is defined by $-\int_X \rho^{(N-1)/N} \,d\fm$ for $N \ge 1$,
it is natural to drop the minus sign for $N<0$ since the function $h(s)=s^{(N-1)/N}$ is convex.

We modify the function $\sigma^{(t)}_{K/N}$ used to characterize the $(K,N)$-convexity as follows:
\[ \tau^{(t)}_{K,N}(\theta) :=t^{1/N} \sigma^{(t)}_{K/(N-1)}(\theta)^{(N-1)/N}
 =t^{1/N} \left( \frac{\fs_{K/(N-1)}(t\theta)}{\fs_{K/(N-1)}(\theta)} \right)^{(N-1)/N} \]
for $t \in (0,1]$ and $\theta>0$ if $K \ge 0$ and for $\theta \in (0,\pi\sqrt{(N-1)/K})$ if $K<0$.
Set also $\tau^{(0)}_{K,N}(\theta):=0$.
Moreover, when $K<0$, we define for convenience
$\sigma^{(t)}_{K/N}(\theta):=\infty$ if $\theta \ge \pi\sqrt{N/K}$
and accordingly $\tau^{(t)}_{K,N}(\theta):=\infty$ if $\theta \ge \pi\sqrt{(N-1)/K}$.

\begin{definition}[Curvature-dimension condition]\label{df:CD}
Let $K \in \R$ and $N<0$.
A metric measure space $(X,d,\fm)$ is said to satisfy the \emph{curvature-dimension condition}
$\CD(K,N)$ if any pair of absolutely continuous measures $\mu_0=\rho_0 \fm, \mu_1=\rho_1 \fm \in \cP^2(X)$
admits a minimal geodesic $(\mu_t)_{t \in [0,1]} \subset \cP^2(X)$ with respect to $W_2$
and an optimal coupling $\pi \in \Pi(\mu_0,\mu_1)$ such that
\begin{equation}\label{eq:CD}
S_{N'}(\mu_t) \le \int_{X \times X} \left\{
 \tau^{(1-t)}_{K,N'}\big( d(x,y) \big) \rho_0(x)^{-1/N'}
 +\tau^{(t)}_{K,N'}\big( d(x,y) \big) \rho_1(y)^{-1/N'}  \right\} \pi(dxdy)
\end{equation}
holds for all $t \in [0,1]$ and $N' \in [N,0)$.
\end{definition}

We remark that \eqref{eq:CD} becomes trivial if $K<0$ and
\[ \pi \left( \{(x,y) \,|\, d(x,y) \ge \pi\sqrt{(N'-1)/K}\} \right) >0. \]
The following variant along \cite{BS} turns out meaningful.

\begin{definition}[Reduced curvature-dimension condition]\label{df:CD*}
A metric measure space $(X,d,\fm)$ is said to satisfy the \emph{reduced curvature-dimension condition}
$\CD^*(K,N)$ if
\begin{equation}\label{eq:CD*}
S_{N'}(\mu_t) \le \int_{X \times X} \left\{
 \sigma^{(1-t)}_{K/N'}\big( d(x,y) \big) \rho_0(x)^{-1/N'}
 +\sigma^{(t)}_{K/N'}\big( d(x,y) \big) \rho_1(y)^{-1/N'}  \right\} \pi(dxdy)
\end{equation}
holds instead of \eqref{eq:CD} in Definition~\ref{df:CD}.
\end{definition}

For $K=0$, \eqref{eq:CD} and \eqref{eq:CD*} coincide
and induce the convexity of $S_{N'}$:
\[ S_{N'}(\mu_t) \le (1-t)S_{N'}(\mu_0)+tS_{N'}(\mu_1). \]
Letting $N \to -\infty$ (in an appropriate way),
both \eqref{eq:CD} and \eqref{eq:CD*} recover $\CD(K,\infty)$:
\[ \Ent_{\fm}(\mu_t) \le (1-t)\Ent_{\fm}(\mu_0) +t\Ent_{\fm}(\mu_1)
 -\frac{K}{2}(1-t)tW_2(\mu_0,\mu_1)^2, \]
where $\Ent_{\fm}(\mu)$ is the \emph{relative entropy} with respect to $\fm$
defined by
\[ \Ent_{\fm}(\mu):=\int_X \rho\log\rho \,d\fm \]
if $\mu=\rho\fm \ll \fm$ and $\int_{\{\rho>1\}} \rho\log\rho \,d\fm<\infty$,
$\Ent_{\fm}(\mu):=\infty$ otherwise.

\begin{remark}\label{rm:hf}
By pioneering work~\cite{JKO} and more generally \cite{AGShf},
heat flow is regarded as the gradient flow of the relative entropy in the Wasserstein space.
Thus, in \cite{EKS}, an expansion bound of heat flow is obtained from $\CD^e(K,N)$
and implies the \emph{Bakry--Ledoux gradient estimate} via the duality argument.
For $N<0$, however, we have an expansion bound of the gradient flow
of a $(K,N)$-convex function only under the Lipschitz continuity (recall Theorem~\ref{th:exp}),
which is never satisfied by the relative entropy.
\end{remark}

In \cite[Proposition~2.5(i)]{BS}, it is shown that $\CD(K,N)$ implies $\CD^*(K,N)$ for $N \ge 1$.
The analogous property holds true for $N<0$.

\begin{proposition}[$\CD(K,N)$ implies $\CD^*(K,N)$]\label{pr:CD}
If $(X,d,\fm)$ satisfies $\CD(K,N)$ for some $K \in \R$ and $N<0$,
then it also satisfies $\CD^*(K,N)$.
\end{proposition}

\begin{proof}
It is sufficient to prove that \eqref{eq:CD} implies \eqref{eq:CD*}
by comparing the coefficient functions $\tau^{(t)}_{K,N'}$ and $\sigma^{(t)}_{K/N'}$
(notice that $N/K<(N-1)/K$ if $K<0$).
For $\theta$ in the domain of $\sigma^{(t)}_{K/N}$ and $N' \in [N,0)$,
we deduce from \cite[Lemma~1.2]{StII} that
\[ \sigma^{(t)}_{(-K)/(1-N')}(\theta)^{1-N'}
 \le \sigma^{(t)}_0(\theta) \sigma^{(t)}_{(-K)/(-N')}(\theta)^{-N'}
 =t\sigma^{(t)}_{K/N'}(\theta)^{-N'}. \]
Hence we have
\[ \tau^{(t)}_{K,N'}(\theta)^{N'} =t\sigma^{(t)}_{K/(N'-1)}(\theta)^{N'-1}
 \ge \sigma^{(t)}_{K/N'}(\theta)^{N'}. \]
Therefore $\tau^{(t)}_{K,N'}(\theta) \le \sigma^{(t)}_{K/N'}(\theta)$
and \eqref{eq:CD} implies \eqref{eq:CD*}.
$\qedd$
\end{proof}

Before discussing the relation with the Ricci curvature,
we give a geometric application of the curvature-dimension condition.

\begin{theorem}[Brunn--Minkowski inequality]\label{th:BM}
Let $(X,d,\fm)$ satisfy $\CD(K,N)$ with $K \in \R$ and $N<0$.
Then, for any measurable sets $A_0,A_1 \subset X$ with
$\diam(A_0 \cup A_1)<\pi\sqrt{(N-1)/K}$ if $K<0$, we have
\begin{equation}\label{eq:BM}
\fm[A_t]^{1/N} \le \sup_{x \in A_0,\, y \in A_1} \tau^{(1-t)}_{K,N}\big( d(x,y) \big) \fm[A_0]^{1/N}
 +\sup_{x \in A_0,\, y \in A_1} \tau^{(t)}_{K,N}\big( d(x,y) \big) \fm[A_1]^{1/N}
\end{equation}
for any $t \in [0,1]$, where $A_t$ is the set consisting of $\gamma(t)$ for minimal geodesics
$\gamma:[0,1] \lra X$ satisfying $\gamma(0) \in A_0$ and $\gamma(1) \in A_1$.

Similarly, if $(X,d,\fm)$ satisfies $\CD^*(K,N)$, then we have
\begin{equation}\label{eq:BM*}
\fm[A_t]^{1/N} \le \sup_{x \in A_0,\, y \in A_1} \sigma^{(1-t)}_{K/N}\big( d(x,y) \big) \fm[A_0]^{1/N}
 +\sup_{x \in A_0,\, y \in A_1} \sigma^{(t)}_{K/N}\big( d(x,y) \big) \fm[A_1]^{1/N}
\end{equation}
for $A_0,A_1 \subset X$ with $\diam(A_0 \cup A_1)<\pi\sqrt{N/K}$ if $K<0$.
\end{theorem}

\begin{proof}
As the proofs are completely the same, we consider only \eqref{eq:BM}.
There is nothing to prove if $\fm[A_0]=0$ or $\fm[A_1]=0$.
If $0<\fm[A_0], \fm[A_1]<\infty$, then combining \eqref{eq:CD} for
$\mu_i=\fm[A_i]^{-1} \cdot \fm|_{A_i}$ ($i=0,1$) and
\[ S_N(\mu_t) =\int_{\supp \mu_t} \rho_t^{-1/N} \,d\mu_t
 \ge \left( \int_{\supp \mu_t} \rho_t^{-1} \,d\mu_t \right)^{1/N}
 =\fm[\supp \mu_t]^{1/N} \ge \fm[A_t]^{1/N} \]
by Jensen's inequality yields \eqref{eq:BM}.
This is enough to conclude also in the case where $\fm[A_0]=\infty$ or $\fm[A_1]=\infty$
by choosing increasing subsets of $A_0$ or $A_1$ and taking the limit of \eqref{eq:BM}.
$\qedd$
\end{proof}

Observe that \eqref{eq:BM} is a lower bound of $\fm[A_t]$ since $N<0$.
On the Euclidean space $\R^n$ equipped with the standard metric,
we take $K=0$ and \eqref{eq:BM} coincides with the \emph{$1/N$-concavity}
of the measure $\fm=\e^{-\psi}\fL^n$ ($\fL^n$ is the Lebesgue measure):
\[ \fm[A_t]^{1/N} \le (1-t)\fm[A_0]^{1/N} +t\fm[A_1]^{1/N}, \]
which is equivalent to the \emph{$p$-concavity} of the function $w=\e^{-\psi}$:
\[ w\big( (1-t)x+ty \big)^p \le (1-t)w(x)^p +tw(y)^p \]
with $p=1/(N-n)$ (see \cite{Bo}, \cite{BrLi}, and \cite[Theorem~1.1]{MR}).
Indeed, when $\psi \in C^2(\R^n)$, the $p$-concavity can be rewritten by calculation
into the weighted Ricci curvature bound:
\[ \Hess\psi -\frac{\nabla\psi \otimes \nabla\psi}{N-n} \ge 0. \]

\begin{remark}\label{rm:BM}
For $N \ge 1$, the Brunn--Minkowski inequality \eqref{eq:BM} implies
the \emph{Bishop--Gromov type volume growth bound}:
\[ \frac{\fm[B(x,R')]}{\fm[B(x,R)]}
 \le \frac{\int_0^{R'} \fs_{K/(N-1)}(r)^{N-1} \,dr}{\int_0^{R} \fs_{K/(N-1)}(r)^{N-1} \,dr} \]
for $0<R \le R'$ ($\le \pi\sqrt{(N-1)/K}$ if $K>0$),
where $B(x,R)$ is the open ball with center $x$ and radius $R$.
This is done by choosing $A_0=\{x\}$, $A_1=B(x,R')$ and $t=R/R'$.
For $N<0$, however, a similar bound can not be expected since $\fm[\{x\}]^{1/N}=\infty$.
For the same reasoning, the \emph{measure contraction property}
does not have a version of $N<0$ (see \cite{Omcp}, \cite[\S 5]{StII}).
\end{remark}

Although $\CD^*(K,N)$ is weaker than $\CD(K,N)$ by calculation,
they are equivalent infinitesimally and characterize a lower Ricci curvature bound
for Riemannian manifolds similarly to the $N \ge 1$ case.

\begin{theorem}\label{th:CD}
Let $(M,g)$ be an $n$-dimensional Riemannian manifold
and fix a measure $\fm=\e^{-\psi}\vol_g$ with $\psi \in \cC^{\infty}(M)$.
Then, given $K \in \R$ and $N<0$, the following are equivalent$:$
\begin{enumerate}[{\rm (I)}]
\item $\Ric_N \ge K$ holds in the sense that $\Ric_N(v) \ge K|v|^2$ for all $v \in TM$.
\item $(M,d_g,\fm)$ satisfies $\CD(K,N)$.
\item $(M,d_g,\fm)$ satisfies $\CD^*(K,N)$.
\end{enumerate}
\end{theorem}

\begin{proof}
The proof is along the same line as the case of $N \in [n,\infty]$,
thus we give only a sketch.
We refer to \cite[\S 8.2]{Oint} and \cite[Proposition~5.5]{BS} for detailed calculations.

(I) $\Rightarrow$ (II):
In the present situation, there is a unique minimal geodesic
$(\mu_t)_{t \in [0,1]} \subset \cP^2(M)$ written as
\[ \mu_t=\rho_t \fm=(\bT_t)_{\sharp}\mu_0, \qquad
 \bT_t(x):=\exp\big( t\nabla\varphi(x) \big) \]
for some $\mu_0$-almost everywhere twice differentiable function $\varphi$,
where $(\bT_t)_{\sharp}\mu_0$ denotes the push-forward of $\mu_0$ by the map $\bT_t$
(see \cite{FG}).
An optimal coupling is also unique and given by $\pi=(\id_M \times \bT_1)_{\sharp}\mu_0$.

Fix $x \in M$ at where $\rho_0(x)>0$, $\varphi$ is twice differentiable
and $\nabla\varphi(x) \neq 0$.
Put $v:=\nabla\varphi(x)$ and $\gamma(t):=\bT_t(x)=\exp(tv)$ for brevity.
Let
\[ \bJ_t(x):=\e^{\psi(x)-\psi(\bT_t(x))} \det[d\bT_t(x)] \]
be the Jacobian of $\bT_t$ with respect to the measure $\fm$.
Then the \emph{Jacobian equation} (or the \emph{Monge--Amper\`e equation})
\begin{equation}\label{eq:MA}
\rho_0(x)=\rho_t\big( \bT_t(x) \big) \bJ_t(x)
\end{equation}
holds.
We take an orthonormal basis $\{e_i\}_{i=1}^n$ of $T_xM$
such that $e_n=v/|v|$ and extend it to the Jacobi fields
$E_i(t):=(d\bT_t)_x(e_i) \in T_{\gamma(t)}M$.
Consider the $n \times n$ matrices $A(t)=(a_{ij}(t))$ and $B(t)=(b_{ij}(t))$ defined by
\[ a_{ij}(t):=\langle E_i(t),E_j(t) \rangle, \qquad
 \nabla_t E_i(t) =\sum_{j=1}^n b_{ij}(t) E_j(t). \]
Note that $\det[d\bT_t(x)]=\sqrt{\det[A(t)]}$ and
$B(t)$ is a symmetric matrix
(see, for example, \cite[(c) in p.~368]{Vi}, \cite[\S 3.1]{OSbw}).
By virtue of the Riccati equation $B'=-R A^{-1} -B^2$
with $R:=(\langle R(E_i,\dot{\gamma})\dot{\gamma},E_j \rangle)$,
we obtain $b'_{nn}=-\sum_{i=1}^n b_{in}^2 \le -b_{nn}^2$.
Hence, when we put $\beta(t):=1+\int_0^t b_{nn} \,ds$,
$\e^{\beta}$ is concave in $t$ and thus
\begin{equation}\label{eq:beta}
\e^{\beta(t)} \ge (1-t)\e^{\beta(0)} +t\e^{\beta(1)}
\end{equation}
holds.
Now we consider the functions
\[ \Phi(t):=\log\left[ \sqrt{\det[A(t)]} \right], \qquad \alpha:=\Phi-\beta \]
and observe from $\Phi'=\trace B$ that
\[ \alpha'' \le -\Ric(\dot{\gamma}) -\frac{(\alpha')^2}{n-1}. \]
Therefore we find $[\e^{\alpha/(n-1)}]'' \e^{-\alpha/(n-1)} \le -\Ric(\dot{\gamma})/(n-1)$.
Hence, by setting
\[ h(t):=\{ \e^{-\psi(x)} \bJ_t(x) \}^{1/N},\quad
 h_1(t):=\e^{-\psi(\gamma(t))/(N-n)},\quad
 h_2:=h_1^{(N-n)/(N-1)} \e^{\alpha/(N-1)}, \]
we have
\[ (N-1)h_2^{-1} h''_2
 \le (N-n)h_1^{-1} h''_1 +(n-1)\e^{-\alpha/(n-1)} [\e^{\alpha/(n-1)}]''
 \le -\Ric_N(\dot{\gamma}) \]
(we remark that the first inequality does not hold if $N \in (1,n)$).
This shows that the function
\[ \frac{h_2(t)-\fc_{K/(N-1)}(t|v|) h_2(0)}{\fs_{K/(N-1)}(t|v|)} \]
is non-decreasing in $t$.
Thus we have
\[ h_2(t) \le \frac{\fs_{K/(N-1)}((1-t)|v|)}{\fs_{K/(N-1)}(|v|)}h_2(0)
 +\frac{\fs_{K/(N-1)}(t|v|)}{\fs_{K/(N-1)}(|v|)}h_2(1). \]
Together with \eqref{eq:beta} and the (reverse) H\"older inequality (see \cite[Claim~4.2]{OT1}),
this yields
\[ h(t)=h_2(t)^{(N-1)/N} (\e^{\beta(t)})^{1/N}
 \le \tau^{(1-t)}_{K,N}(|v|) h(0) +\tau^{(t)}_{K,N}(|v|) h(1), \]
which is equivalent to the convexity of the $(1/N)$-th power of Jacobian:
\begin{equation}\label{eq:Jac}
\bJ_t(x)^{1/N} \le \tau^{(1-t)}_{K,N}(|v|) +\tau^{(t)}_{K,N}(|v|) \bJ_1(x)^{1/N}.
\end{equation}

Integrating this infinitesimal inequality \eqref{eq:Jac} immediately gives \eqref{eq:CD} with $N'=N$.
Precisely, by virtue of the Jacobian equation \eqref{eq:MA}, we obtain
\begin{align*}
S_N(\mu_t) &=\int_M (\rho_t \circ \bT_t)^{(N-1)/N} \bJ_t \,d\fm
 =\int_M \left( \frac{\bJ_t}{\rho_0} \right)^{1/N} \,d\mu_0 \\
&\le \int_M \left\{ \frac{\tau^{(1-t)}_{K,N}(d_g(x,\bT_1(x)))}{\rho_0(x)^{1/N}}
 +\frac{\tau^{(t)}_{K,N}(d_g(x,\bT_1(x)))}{\rho_1(\bT_1(x))^{1/N}} \right\} \mu_0(dx) \\
&= \int_{M \times M} \left\{ \frac{\tau^{(1-t)}_{K,N}(d_g(x,y))}{\rho_0(x)^{1/N}}
 +\frac{\tau^{(t)}_{K,N}(d_g(x,y))}{\rho_1(y)^{1/N}} \right\} \pi(dxdy).
\end{align*}
This completes the proof since $\Ric_{N'} \ge \Ric_N$ for $N' \in [N,0)$.

(II) $\Rightarrow$ (III):
This was shown in Proposition~\ref{pr:CD}.

(III) $\Rightarrow$ (I):
Fix a unit vector $v \in T_xM$ and let $\gamma:(-\delta,\delta) \lra M$
be the geodesic with $\dot{\gamma}(0)=v$.
Put $a=\langle \nabla\psi,v \rangle/(N-n)$ and consider the open balls
\[ A_0:=B\big( \gamma(-r),\ve(1+ar) \big), \qquad
 A_1:=B\big( \gamma(r),\ve(1-ar) \big) \]
for $0<\ve \ll r \ll \delta$.
It follows from \eqref{eq:BM*} with $t=1/2$ that
\[ \fm[A_{1/2}]^{1/N} \le \sigma^{(1/2)}_{K/N}\big( 2r+O(\ve) \big)
 \left( \fm[A_0]^{1/N} +\fm[A_1]^{1/N} \right). \]
Observe that (see \cite[Proposition~5.5]{BS})
\[ \sigma^{(1/2)}_{K/N}(2r) =\frac{1}{2}+\frac{K}{4N}r^2 +O(r^4). \]
By combining this with the asymptotic behaviors of $\fm[A_0]$, $\fm[A_1]$ and $\fm[A_{1/2}]$
in terms of the weight function $\psi$ and the Ricci curvature,
we can conclude $\Ric_N(v) \ge K$.
$\qedd$
\end{proof}

This characterization is generalized to Finsler manifolds verbatim,
see \cite{Oint} for the case of $N \in [n,\infty]$.

\begin{remark}[Lott and Villani's version of $\CD(K,N)$]\label{rm:LVCD}
From the infinitesimal inequality \eqref{eq:Jac}, we further obtain
Lott and Villani's version of the curvature-dimension condition studied
in \cite{LV1}, \cite{LV2} independently of Sturm's work.
Their version extends the class of entropies to the ones induced from functions
in \emph{displacement convexity classes} $\DC_N$.
For $N \ge 1$, McCann~\cite{Mc1} introduced $\DC_N$
as the set of all continuous convex functions $u:[0,\infty) \lra \R$ such that $u(0)=0$
and that $\phi(s):=s^N u(s^{-N})$ is convex on $(0,\infty)$.
We adopted the same definition for $N<0$ in \cite{OT2},
then Lott and Villani's version of $\CD(K,N)$ means that
\begin{align*}
\int_M u(\rho_t) \,d\fm
&\le (1-t) \int_{M \times M} \frac{\beta^{(1-t)}_{K,N}(d_g(x,y))}{\rho_0(x)}
 u\left( \frac{\rho_0(x)}{\beta^{(1-t)}_{K,N}(d_g(x,y))} \right) \pi(dxdy) \\
&\quad +t\int_{M \times M} \frac{\beta^{(t)}_{K,N}(d_g(x,y))}{\rho_1(x)}
 u\left( \frac{\rho_1(x)}{\beta^{(t)}_{K,N}(d_g(x,y))} \right) \pi(dxdy),
\end{align*}
where $\beta^{(t)}_{K,N}(\theta):=\{ \tau^{(t)}_{K,N}(\theta)/t \}^N$.
This follows from \eqref{eq:Jac} and the non-decreasing property of $\phi$ (see \cite[Lemma~3.2]{OT2}).
Choosing $u(r)=Nr(1-r^{-1/N})$ recovers \eqref{eq:CD}.
\end{remark}

An estimate similar to the proof of Theorem~\ref{th:BW}
gives the following examples of $\CD(K,N)$-spaces.
Compare Corollary~\ref{cr:weit} with \cite[Proposition~3.3]{EKS} and Lemma~\ref{lm:sum}.

\begin{corollary}[Weighted spaces]\label{cr:weit}
Let $K_1,K_2 \in \R$, $N_2 \ge n$ and $N_1<-N_2$.
If an $n$-dimensional weighted Riemannian manifold $(M,d_g,\fm)$ satisfies $\CD(K_2,N_2)$
and $\Psi \in \cC^2(M)$ is $(K_1,N_1)$-convex,
then $(M,d_g,\e^{-\Psi}\fm)$ satisfies $\CD(K_1+K_2,N_1+N_2)$.
\end{corollary}

\begin{proof}
Put $\fm=\e^{-\psi}\vol_g$, $K=K_1+K_2$ and $N=N_1+N_2$.
We remark that $N_2=n$ only if $\psi$ is constant.
The weighted Ricci curvature $\overline{\Ric}_N(v)$ with respect to the measure $\e^{-\Psi}\fm$
is bounded by using $\Ric_{N_2}(v)$ for $\fm$ as
\begin{align*}
\overline{\Ric}_N(v)
&= \Ric_{N_2}(v) +\Hess\Psi(v,v) +\frac{\langle \nabla\psi,v \rangle^2}{N_2-n}
 -\frac{\langle \nabla(\psi+\Psi),v \rangle^2}{N-n} \\
&\ge K|v|^2 +\frac{\langle \nabla\Psi,v \rangle^2}{N_1}
 +\frac{\langle \nabla\psi,v \rangle^2}{N_2-n} -\frac{\langle \nabla(\psi+\Psi),v \rangle^2}{N-n} \\
&\ge K|v|^2.
\end{align*}
This completes the proof.
$\qedd$
\end{proof}

For example, by Example~\ref{ex:KN}(c),
$((0,\infty),|\cdot|,x^N dx)$ with the Euclidean distance $|\cdot|$ satisfies $\CD(0,N)$ for $N<0$.

\begin{corollary}[Product spaces]\label{cr:prod}
Let $K \in \R$, $N_2 \ge n_2$ and $N_1<-N_2$.
If $n_i$-dimensional weighted Riemannian manifolds $(M_i,d_{g_i},\fm_i)$ satisfy $\CD(K,N_i)$ for $i=1,2$,
then the Cartesian product $(M_1 \times M_2,d_{g_1 \times g_2},\fm_1 \times \fm_2)$ satisfies $\CD(K,N_1+N_2)$.
\end{corollary}

\begin{proof}
Put $\fm_i=\e^{-\psi_i}\vol_{g_i}$ and $N=N_1+N_2$.
Then, for $v=(v_1,v_2) \in TM_1 \times TM_2$, we have
\begin{align*}
\Ric_N(v)
&= \sum_{i=1}^2 \{ \Ric(v_i)+\Hess\psi_i(v_i,v_i) \}
 -\frac{(\langle \nabla\psi_1,v_1 \rangle +\langle \nabla\psi_2,v_2 \rangle)^2}{N-(n_1+n_2)} \\
&\ge \Ric_{N_1}(v_1) +\Ric_{N_2}(v_2)
 \ge K(|v_1|^2 +|v_2|^2).
\end{align*}
$\qedd$
\end{proof}

\subsection{Entropic curvature-dimension condition and functional inequalities}\label{ssc:CDe}

We finally introduce another version of the curvature-dimension condition
in terms of the $(K,N)$-convexity studied in previous sections.
This notion has applications to functional inequalities similarly to the $N>0$ case in \cite{EKS}
(see also the original case of $N=\infty$ by Otto and Villani~\cite{OV}).

Let $(X,d,\fm)$ be a complete, separable metric measure space, and assume
\[ \int_X \e^{-cd(x,y)^2} \,\fm(dy)<\infty \]
for some (and hence all) $x \in X$ and all $c>0$.
This hypothesis ensures that $\Ent_{\fm}$ is never being $-\infty$ on $\cP^2(X)$
and is lower semi-continuous with respect to $W_2$.

\begin{definition}[Entropic curvature-dimension condition]\label{df:CDe}
Let $K \in \R$ and $N<0$.
A metric measure space $(X,d,\fm)$ is said to satisfy
the \emph{entropic curvature-dimension condition} $\CD^e(K,N)$
if the relative entropy $\Ent_{\fm}$ is $(K,N)$-convex on $(\cP^2(X),W_2)$.
\end{definition}

This condition was introduced in \cite{EKS} for $N>0$,
and turned out equivalent to $\CD^*(K,N)$ for \emph{essentially non-branching} spaces
in the sense of \cite[Definition~3.10]{EKS}
such as Riemannian or Finsler manifolds and Alexandrov spaces (\cite[Theorem~3.12]{EKS}).
Therefore Riemannian or Finsler manifolds with $\Ric_{\infty} \ge K$ satisfy $\CD^e(K,N)$ for all $N<0$.
For $N<0$, however, a similar argument shows only that $\CD^e(K,N)$ implies $\CD^*(K,N)$.

\begin{proposition}[$\CD^e(K,N)$ implies $\CD^*(K,N)$]\label{pr:CDe}
Let $(X,d,\fm)$ be essentially non-branching.
If $(X,d,\fm)$ satisfies $\CD^e(K,N)$ for some $K \in \R$ and $N<0$,
then it also satisfies $\CD^*(K,N)$.
\end{proposition}

\begin{proof}
We give only a sketchy proof.
By using the convex function $G_t$ appearing in the proof of Lemma~\ref{lm:sum},
$\CD^e(K,N)$ is written as
\[ -\frac{1}{N}\Ent_{\fm}(\mu_t)
 \le G_t\left( -\frac{\Ent_{\fm}(\mu_0)}{N}, -\frac{\Ent_{\fm}(\mu_1)}{N},
 \frac{K}{N} W_2(\mu_0,\mu_1)^2 \right). \]
Jensen's inequality then yields
\[ -\frac{1}{N}\Ent_{\fm}(\mu_t)
 \le \int_{X \times X} G_t\left( -\frac{\log\rho_0(x)}{N}, -\frac{\log\rho_1(y)}{N},
 \frac{K}{N} d(x,y)^2 \right) \pi(dxdy), \]
which implies the infinitesimal version of $\CD^*(K,N)$:
\begin{equation}\label{eq:locCD*}
\rho_t\big( \gamma(t) \big)^{-1/N} \le \sigma^{(1-t)}_{K/N}\big( d(x,y) \big) \rho_0(x)^{-1/N}
 +\sigma^{(t)}_{K/N}\big( d(x,y) \big) \rho_1(y)^{-1/N}
\end{equation}
via the localization based on the non-branching property
(see (iii) $\Rightarrow$ (ii) of \cite[Theorem~3.12]{EKS}).
Finally the integration gives $\CD^*(K,N)$.
$\qedd$
\end{proof}

\begin{remark}\label{rm:CDe}
One sees from the usage of Jensen's inequality in Proposition~\ref{pr:CDe} that
the inequality \eqref{eq:locCD*} does not imply $\CD^e(K,N)$.
In other words, $\CD^e(K,N)$ as an integrated inequality
is stronger than its infinitesimal version \eqref{eq:locCD*}.
In fact, it seems that $\Ric_N \ge K$ does not imply $\CD^e(K,N)$ for $N<0$.
This is because, according to the notations in Theorem~\ref{th:CD},
\[ \Ent_{\fm}(\mu_t) =\Ent_{\fm}(\mu_0) -\int_M \log\bJ_t \,d\mu_0 \]
and the implication from $\Ric_N \ge K$ to $\CD^e(K,N)$ for $N \ge n$ is verified
by the calculations:
\begin{align*}
\left( -\int_M \log\bJ_t \,d\mu_0 \right)''
&= \int_M \{ \Ric_{\infty}(\dot{\gamma}) +\trace(B^2) \}(t) \,d\mu_0 \\
&\ge \int_M \left\{ K|\dot{\gamma}|^2 +\frac{\langle \nabla\psi,\dot{\gamma} \rangle^2}{N-n}
 +\frac{(\trace B)^2}{n} \right\}(t) \,d\mu_0 \\
&\ge KW_2(\mu_0,\mu_1)^2
 +\int_M \frac{(\langle \nabla\psi,\dot{\gamma} \rangle -\trace B)^2}{N}(t) \,d\mu_0
\end{align*}
and
\begin{align*}
\int_M \frac{(\langle \nabla\psi,\dot{\gamma} \rangle -\trace B)^2}{N}(t) \,d\mu_0
& \ge \frac{1}{N}\left( \int_M (\langle \nabla\psi,\dot{\gamma} \rangle -\trace B)(t) \,d\mu_0 \right)^2 \\
&= \frac{1}{N} \left\{ \left( -\int_M \log\bJ_t \,d\mu_0 \right)' \right\}^2.
\end{align*}
The last inequality is the Cauchy--Schwarz inequality for which $N>0$ is necessary.
\end{remark}

From here on, we set
\[ E_N(\mu):=\exp\left( -\frac{\Ent_{\fm}(\mu)}{N} \right). \]
The condition $\CD^e(K,N)$ implies a variant of the HWI inequality
similarly to \cite[Theorem~3.28]{EKS}.
Define the (relative) \emph{Fisher information} of $\mu \in \cP^2(X)$
with respect to the reference measure $\fm$ by
\[ I_{\fm}(\mu):=|\grad \Ent_{\fm}|(\mu)^2. \]
Under mild assumptions on the space $(X,d,\fm)$ and an absolutely continuous measure
$\mu=\rho\fm$, we have
\[ I_{\fm}(\mu)=\int_X \frac{|\grad \rho|^2}{\rho} \,d\fm. \]
This representation is one of the key ingredients in the identification of two gradient flows
regarded as heat flow (see \cite{GKO}, \cite{AGShf}).

\begin{theorem}[$N$-HWI inequality]\label{th:HWI}
Let $(X,d,\fm)$ satisfy $\CD^e(K,N)$ for $K \in \R$ and $N<0$.
Then, for any $\mu_0,\mu_1 \in \cP^2(X)$ with $W_2(\mu_0,\mu_1) \le \pi\sqrt{N/K}$ if $K<0$,
we have
\begin{equation}\label{eq:HWI}
\frac{E_N(\mu_1)}{E_N(\mu_0)} \ge \fc_{K/N}\big( W_2(\mu_0,\mu_1) \big)
 +\frac{\fs_{K/N}(W_2(\mu_0,\mu_1))}{N} \sqrt{I_{\fm}(\mu_0)}.
\end{equation}
\end{theorem}

\begin{proof}
Let $(\mu_t)_{t \in [0,1]}$ be a minimal geodesic along which
the $(K,N)$-convexity inequality \eqref{eq:wKN+} holds.
Arguing as in Lemma~\ref{lm:KN}(ii) $\Rightarrow$ (iii)
and setting $W_2:=W_2(\mu_0,\mu_1)$ for brevity, we have
\begin{align*}
E_N(\mu_1) &\ge \fc_{K/N}(W_2) E_N(\mu_0)
 +\frac{\fs_{K/N}(W_2)}{W_2} \limsup_{t \downarrow 0}\frac{E_N(\mu_t)-E_N(\mu_0)}{t} \\
&= \fc_{K/N}(W_2) E_N(\mu_0)
 -\frac{\fs_{K/N}(W_2)}{W_2} \frac{E_N(\mu_0)}{N}
 \limsup_{t \downarrow 0}\frac{\Ent_{\fm}(\mu_t)-\Ent_{\fm}(\mu_0)}{t}.
\end{align*}
Then we deduce \eqref{eq:HWI} from
\begin{align*}
\limsup_{t \downarrow 0}\frac{\Ent_{\fm}(\mu_t)-\Ent_{\fm}(\mu_0)}{t}
&\ge -\liminf_{t \downarrow 0}\frac{\max\{\Ent_{\fm}(\mu_0)-\Ent_{\fm}(\mu_t),0\}}{t} \\
&\ge -|\grad \Ent_{\fm}|(\mu_0) W_2(\mu_0,\mu_1).
\end{align*}
$\qedd$
\end{proof}

In the case where $K>0$ and $\fm \in \cP^2(X)$, \eqref{eq:HWI} implies
the following generalizations of well known inequalities
(see \cite[Corollaries~3.31, 3.29]{EKS} for the $N>0$ case).
Note that $\Ent_{\fm}$ is nonnegative in this case by Jensen's inequality.

\begin{corollary}[$N$-Talagrand inequality]\label{cr:Tal}
Assume that $\fm \in \cP^2(X)$ and $(X,d,\fm)$ satisfies $\CD^e(K,N)$
with $K>0$ and $N<0$.
Then, for any $\mu \in \cP^2(X)$, we have
\[ \Ent_{\fm}(\mu) \ge -N\log\bigg[ \cosh\bigg( \sqrt{-\frac{K}{N}} W_2(\fm,\mu) \bigg) \bigg]. \]
\end{corollary}

This is a nontrivial estimate since the RHS is nonnegative.

\begin{proof}
We apply \eqref{eq:HWI} to $\mu_0=\fm$ and $\mu_1=\mu$.
Since $\Ent_{\fm}(\fm)=I_{\fm}(\fm)=0$, we find
\[ E_N(\mu) \ge \fc_{K/N}\big( W_2(\fm,\mu) \big)
 =\cosh \bigg( \sqrt{-\frac{K}{N}} W_2(\fm,\mu) \bigg). \]
This is equivalent to the desired inequality.
$\qedd$
\end{proof}

\begin{corollary}[$N$-logarithmic Sobolev inequality]\label{cr:Sob}
Let $(X,d,\fm)$ be as in Corollary~$\ref{cr:Tal}$.
Then, for any $\mu \in \cP^2(X)$ satisfying
\begin{equation}\label{eq:Soba}
\fc_{K/N}\big( W_2(\mu,\fm) \big) +\frac{\fs_{K/N}(W_2(\mu,\fm))}{N} \sqrt{I_{\fm}(\mu)}>0,
\end{equation}
we have
\[ KN\bigg\{ \exp\left( \frac{2}{N}\Ent_{\fm}(\mu) \right) -1 \bigg\} \le I_{\fm}(\mu). \]
\end{corollary}

Observe that $\exp(2\Ent_{\fm}(\mu)/N) \le 1$ and hence the LHS is nonnegative.

\begin{proof}
We apply \eqref{eq:HWI} in the reverse direction, namely $\mu_0=\mu$ and $\mu_1=\fm$.
This yields
\[ \exp\left( \frac{\Ent_{\fm}(\mu)}{N} \right) \ge \fc +\frac{\fs}{N}\sqrt{I_{\fm}(\mu)}, \]
where we abbreviated as $\fc:=\fc_{K/N}(W_2(\mu,\fm))$ and
$\fs:=\fs_{K/N}(W_2(\mu,\fm))$.
The RHS is positive by assumption, thus we have
\[ \exp\left( \frac{2}{N}\Ent_{\fm}(\mu) \right)
 \ge \fc^2 +\frac{2}{N} \fc \fs \sqrt{I_{\fm}(\mu)} +\frac{\fs^2}{N^2}I_{\fm}(\mu). \]
Since
\[ -\frac{2}{-N}\fs \fc \sqrt{I_{\fm}(\mu)}
 \ge -\bigg\{ K\left( \frac{\fs}{\sqrt{-N}} \right)^2
 +K^{-1} \left( \frac{\fc \sqrt{I_{\fm}(\mu)}}{\sqrt{-N}} \right)^2 \bigg\}
 =\frac{K \fs^2}{N} +\frac{\fc^2 I_{\fm}(\mu)}{KN}, \]
we obtain
\[ \exp\left( \frac{2}{N}\Ent_{\fm}(\mu) \right)
 \ge \left( \fc^2 +\frac{K}{N}\fs^2 \right) \left( 1+\frac{I_{\fm}(\mu)}{KN} \right)
 =1+\frac{I_{\fm}(\mu)}{KN} \]
and complete the proof.
$\qedd$
\end{proof}

The assumption \eqref{eq:Soba} is rewritten as
\[ N<-\frac{\fs_{K/N}(W_2(\mu,\fm))}{\fc_{K/N}(W_2(\mu,\fm))} \sqrt{I_{\fm}(\mu)}, \]
which is achieved by letting $N$ smaller but then $\CD^e(K,N)$ is getting stronger.

{\small

}


\begin{thebibliography}{AGS3}

\bibitem[AG]{AG}
L.~Ambrosio and N.~Gigli, A user's guide to optimal transport,
Modeling and optimisation of flows on networks, 1--155,
Lecture Notes in Math., {\bf 2062}, Springer, Heidelberg, 2013.

\bibitem[AGS1]{AGSbook}
L.~Ambrosio, N.~Gigli and G.~Savar\'e,
Gradient flows in metric spaces and in the space of probability measures,
Birkh\"auser Verlag, Basel, 2005.

\bibitem[AGS2]{AGShf}
L.~Ambrosio, N.~Gigli and G.~Savar\'e,
Calculus and heat flow in metric measure spaces and applications to spaces with Ricci bounds from below,
Invent.\ Math.\ {\bf 195} (2014), 289--391.

\bibitem[AGS3]{AGSrcd}
L.~Ambrosio, N.~Gigli and G.~Savar\'e,
Metric measure spaces with Riemannian Ricci curvature bounded from below,
Duke Math.\ J.\ {\bf 163} (2014), 1405--1490.

\bibitem[BS]{BS}
K.~Bacher and K.-T.~Sturm,
Localization and tensorization properties of the curvature-dimension condition
for metric measure spaces,
J.\ Funct.\ Anal.\ {\bf 259} (2010), 28--56.

\bibitem[BGL]{BGL}
D.~Bakry, I.~Gentil and M.~Ledoux,
On Harnack inequalities and optimal transportation,
to appear in Ann.\ Scuola Norm.\ Sup.\ Pisa.
Available at {\sf arXiv:1210.4650}

\bibitem[BoLe]{BoLe}
S.~G.~Bobkov and M.~Ledoux,
Weighted Poincar\'e-type inequalities for Cauchy and other convex measures,
Ann.\ Probab.\ {\bf 37} (2009), 403--427.

\bibitem[BGG]{BGG}
F.~Bolley, I.~Gentil and A.~Guillin, Dimensional contraction via Markov transportation distance,
J.\ Lond.\ Math.\ Soc.\ (2) {\bf 90} (2014), 309--332.

\bibitem[Bo]{Bo}
C.~Borell, Convex set functions in $d$-space. 
Period.\ Math.\ Hungar.\ {\bf 6} (1975), 111--136.

\bibitem[BrLi]{BrLi}
H.~J.~Brascamp and E.~H.~Lieb,
On extensions of the Brunn--Minkowski and Pr\'ekopa--Leindler theorems, including inequalities for log concave functions,
and with an application to the diffusion equation,
J.\ Functional Analysis {\bf 22} (1976), 366--389.

\bibitem[DS]{DS}
S.~Daneri and G.~Savar\'e,
Eulerian calculus for the displacement convexity in the Wasserstein distance,
SIAM J.\ Math.\ Anal.\ {\bf 40} (2008), 1104--1122.

\bibitem[EKS]{EKS}
M~Erbar, K.~Kuwada and K.-T.~Sturm,
On the equivalence of the entropic curvature-dimension condition and Bochner's inequality on metric measure spaces,
to appear in Invent.\ Math.
Available at {\sf arXiv:1303.4382}

\bibitem[FG]{FG}
A.~Figalli and N.~Gigli,
Local semiconvexity of Kantorovich potentials on non-compact manifolds,
ESAIM Control Optim.\ Calc.\ Var.\ {\bf 17} (2011), 648--653.

\bibitem[GM]{GM}
N.~Garofalo and A.~Mondino,
Li--Yau and Harnack type inequalities in $RCD^*(K,N)$ metric measure spaces,
Nonlinear Anal.\ {\bf 95} (2014), 721--734.

\bibitem[GKO]{GKO}
N.~Gigli, K.~Kuwada and S.~Ohta, Heat flow on Alexandrov spaces,
Comm.\ Pure Appl.\ Math.\ {\bf 66} (2013), 307--331.

\bibitem[HKX]{HKX}
B.~Hua, M.~Kell and C.~Xia, Harmonic functions on metric measure spaces,
Preprint (2013). Available at {\sf arXiv:1308.3607}

\bibitem[JKO]{JKO}
R.~Jordan, D.~Kinderlehrer and F.~Otto,
The variational formulation of the Fokker--Planck equation,
SIAM J.\ Math.\ Anal.\ {\bf 29} (1998), 1--17.

\bibitem[KM]{KM}
A.~V.~Kolesnikov and E.~Milman,
Poincar\'e and Brunn--Minkowski inequalities on weighted Riemannian manifolds with boundary,
Preprint (2013). Available at {\sf arXiv:1310.2526}

\bibitem[Ku]{Ku}
K.~Kuwada, Space-time Wasserstein controls and Bakry--Ledoux type gradient estimates,
to appear in Calc.\ Var.\ Partial Differential Equations.
Available at {\sf arXiv:1308.5471}

\bibitem[Le]{Le}
P.~W.~Y.~Lee, Displacement interpolations from a Hamiltonian point of view,
J.\ Funct.\ Anal.\ {\bf 265} (2013), 3163--3203.

\bibitem[Lo]{Lo}
J.~Lott, Some geometric properties of the Bakry--\'Emery-Ricci tensor, 
Comment.\ Math.\ Helv.\ {\bf 78} (2003), 865--883.

\bibitem[LV1]{LV1}
J.~Lott and C.~Villani, Weak curvature conditions and functional inequalities,
J.\ Funct.\ Anal.\ {\bf 245} (2007), 311--333.

\bibitem[LV2]{LV2}
J.~Lott and C.~Villani,
Ricci curvature for metric-measure spaces via optimal transport,
Ann.\ of Math.\ (2) {\bf 169} (2009), 903--991.

\bibitem[Mc]{Mc1}
R.~J.~McCann, A convexity principle for interacting gases,
Adv.\ Math.\ {\bf 128} (1997), 153--179.

\bibitem[MR]{MR}
E.~Milman and L.~Rotem,
Complemented Brunn--Minkowski inequalities and isoperimetry for homogeneous and non-homogeneous measures,
Adv.\ Math.\ {\bf 262} (2014), 867--908.

\bibitem[Oh1]{Omcp}
S.~Ohta, On the measure contraction property of metric measure spaces,
Comment.\ Math.\ Helv.\ {\bf 82} (2007), 805--828.

\bibitem[Oh2]{Ogra}
S.~Ohta, Gradient flows on Wasserstein spaces over compact Alexandrov spaces,
Amer.\ J.\ Math.\ {\bf 131} (2009), 475--516.

\bibitem[Oh3]{Oint}
S.~Ohta, Finsler interpolation inequalities,
Calc.\ Var.\ Partial Differential Equations {\bf 36} (2009), 211--249.

\bibitem[Oh4]{Oham}
S.~Ohta, On the curvature and heat flow on Hamiltonian systems,
Anal.\ Geom.\ Metr.\ Spaces {\bf 2} (2014), 81--114.

\bibitem[OS1]{OShf}
S.~Ohta and K.-T.~Sturm, Heat flow on Finsler manifolds,
Comm.\ Pure Appl.\ Math.\ {\bf 62} (2009), 1386--1433.

\bibitem[OS2]{OSnc}
S.~Ohta and K.-T.~Sturm, Non-contraction of heat flow on Minkowski spaces,
Arch.\ Ration.\ Mech.\ Anal.\ {\bf 204} (2012), 917--944.

\bibitem[OS3]{OSbw}
S.~Ohta and K.-T.~Sturm, Bochner--Weitzenb\"ock formula and Li--Yau estimates on Finsler manifolds,
Adv.\ Math.\ {\bf 252} (2014), 429--448.

\bibitem[OT1]{OT1}
S.~Ohta and A.~Takatsu, Displacement convexity of generalized relative entropies,
Adv.\ Math.\ {\bf 228} (2011), 1742--1787.

\bibitem[OT2]{OT2}
S.~Ohta and A.~Takatsu, Displacement convexity of generalized relative entropies.~II,
Comm.\ Anal.\ Geom.\ {\bf 21} (2013), 687--785.

\bibitem[Ot]{Ot}
F.~Otto,
The geometry of dissipative evolution equations: the porous medium equation,
Comm.\ Partial Differential Equations {\bf 26} (2001), 101--174.

\bibitem[OV]{OV}
F.~Otto and C.~Villani,
Generalization of an inequality by Talagrand and links with the logarithmic Sobolev inequality,
J.\ Funct.\ Anal.\ {\bf 173} (2000), 361--400.

\bibitem[Qi]{Qi}
Z.~Qian, Estimates for weighted volumes and applications,
Quart.\ J.\ Math.\ Oxford Ser.\ (2) {\bf 48} (1997), 235--242.

\bibitem[St1]{StI}
K.-T.~Sturm, On the geometry of metric measure spaces.~I,
Acta Math.\ {\bf 196} (2006), 65--131.

\bibitem[St2]{StII}
K.-T.~Sturm, On the geometry of metric measure spaces.~II,
Acta Math.\ {\bf 196} (2006), 133--177.

\bibitem[Vi]{Vi}
C.~Villani, Optimal transport, old and new, Springer-Verlag, Berlin, 2009.

\end{thebibliography}
\end{document}